\documentclass[10pt,a4paper]{amsart}
\usepackage[DIV14,paper=a4,BCOR14mm,headinclude]{typearea}
\usepackage{amsaddr}
\usepackage[british]{babel}
\usepackage{amssymb}
\usepackage{amsmath}
\usepackage{graphicx}
\usepackage{color}
\usepackage{tikz}
\usepackage{pgfplots, pgfplotstable}
\usepgfplotslibrary{groupplots}
\usepackage{enumitem}
\usepackage{hyperref}
\usepackage{fp-eval}
\usepackage{esint}
\usepackage{subfigure}
\usepackage{caption}
\usepackage{mathabx}
\usepackage{booktabs}
\definecolor{darkblue}{rgb}{0,0,.7}
\hypersetup{colorlinks=true,linkcolor=darkblue,citecolor=darkblue}

\newlist{alphenum}{enumerate}{1}
\setlist[alphenum]{fullwidth,label={(\alph*)}}

\DeclareMathOperator{\argmin}{argmin}

\DeclareMathOperator{\dx}{\, \mathit{dx}}
\DeclareMathOperator{\ds}{\, \mathit{ds}}
\newcommand{\vecb}[1]{\boldsymbol{#1}}

\newcommand{\IR}{I_{RT_0}}
\newcommand{\IL}{I_{\mathcal{L}}}
\newcommand{\IN}{I_{\mathcal{N}_0}}
\newcommand{\ClemN}{\Pi_{\mathcal{N}_0}}

\usepackage{siunitx}
\newcommand{\numcoef}[1]{\num[round-precision=2,round-mode=places, scientific-notation=true, group-minimum-digits = 6]{#1}}

\theoremstyle{plain}

\newtheorem{theorem}{Theorem}[section]
\newtheorem{lemma}[theorem]{Lemma}
\newtheorem{proposition}[theorem]{Proposition}
\newtheorem{remark}[theorem]{Remark}
\newtheorem{assumption}[theorem]{Assumption}

\newcommand{\TheTitle}{Refined a posteriori error estimation for classical and pressure-robust Stokes finite element methods}

\author{P. L. Lederer}
\address{Institute for Analysis and Scientific Computing, TU Wien}
\email{philip.lederer@tuwien.ac.at}

\author{C. Merdon}
\address{Weierstrass Institute for Applied Analysis and Stochastics, Berlin}
\email{christian.merdon@wias-berlin.de}

\author{J. Sch\"oberl}
\address{Institute for Analysis and Scientific Computing, TU Wien}
\email{joachim.schoeberl@tuwien.ac.at}

\title{{\TheTitle}}
\renewcommand{\shorttitle}{Error control for pressure-robust finite element methods}

\date{5th of December, 2017}
  
\begin{document}

\maketitle

\begin{abstract}

Recent works showed that pressure-robust modifications of mixed finite element methods for the Stokes equations outperform their standard versions in many cases. This is achieved by divergence-free reconstruction operators and results in pressure independent velocity error estimates which are robust with respect to small viscosities. In this paper we develop a posteriori error control which reflects this robustness.

The main difficulty lies in the volume contribution of the standard residual-based approach
that includes the \(L^2\)-norm of the right-hand side. However, the velocity is only steered by the divergence-free part of this source term. An efficient error estimator must approximate this divergence-free part in a proper manner, otherwise it can be dominated by the pressure error.

To overcome this difficulty a novel approach is suggested
that uses arguments from the stream function and vorticity formulation of the Navier--Stokes equations. The novel error estimators only take the $\mathrm{curl}$
of the right-hand side into account and so lead to provably reliable, efficient and pressure-independent upper bounds
in case of a pressure-robust method in particular in pressure-dominant situations. This is also confirmed by some numerical examples with the novel pressure-robust modifications of the Taylor--Hood and mini finite element methods.


\end{abstract}

\section{Introduction}

This paper studies a posteriori error estimators for the velocity of the Stokes equation with a special focus on pressure-robust finite element methods. Pressure-robustness
is closely related to the \(L^2\)-orthogonality of divergence-free functions onto gradients of \(H^1\)-functions. In particular, the exact velocity
\(\vecb{u}\) of the Stokes equations (with zero boundary data),
\begin{align*}
  -\nu \Delta \vecb{u} + \nabla p & = \vecb{f} \text{ in } \Omega \quad \text{and} \quad \vecb{u} \in \vecb{V}_0 := \lbrace \vecb{v} \in H^1_0(\Omega)^2 : \mathrm{div} \vecb{v} = 0 \rbrace,
\end{align*}
is orthogonal onto any \(q \in L^2(\Omega)\) in the sense that \(\int_\Omega q \mathrm{div}(\vecb{u}) \dx = 0\).
Consequently, \(\vecb{u}\) also solves the Stokes equations
with \(\vecb{f}\) replaced by \(\vecb{f} + \nabla q\) for  \(q \in H^1(\Omega)\).
This invariance property is in general not preserved for discretely divergence-free
testfunctions of most classical finite element methods that
relax the divergence constraint to attain inf-sup stability.
With an inf-sup-stable pair of a discrete velocity space \(\vecb{V}_h\)
and some discrete pressure space \(Q_h\) and the discretely divergence-free functions
\(\vecb{V}_{0,h} \subset \vecb{V}_h\),
the consistency error from the relaxed divergence constraint can be
expressed by the discrete dual norm , for any \(q \in L^2(\Omega)\),
\begin{align}\label{intro_discrete_orthogonality}
  \| \nabla q \|_{\vecb{V}_{0,h}^\star}
  := \!\! \sup_{\vecb{v}_h \in V_{0,h} \setminus \lbrace \vecb{0} \rbrace} \!\! \frac{\int_\Omega q \mathrm{div} \vecb{v}_h \dx}{\| \nabla \vecb{v}_h \|_{L^2}}
  \leq \begin{cases}
          \min_{q_h \in Q_h} \| q - q_h \|_{L^2} & \text{if } \vecb{V}_{0,h} \not\subseteq \vecb{V}_0,\\
          0 & \text{if } \vecb{V}_{0,h} \subseteq \vecb{V}_0.
       \end{cases}
\end{align}
Besides some expensive or exotic divergence-free methods like the Scott-Vogelius finite element method \cite{scott:vogelius:conforming,Zhang05}, most classical inf-sup stable
mixed finite element methods, including the popular Taylor--Hood \cite{HT74} and mini finite element families \cite{ABF84} have \(V_{0,h} \not\subseteq \vecb{V}_0\) and so
the term from \eqref{intro_discrete_orthogonality} appears in their a priori velocity gradient error estimate \cite{MR3097958} scaled with \(1/\nu\), i.e.
\begin{align}\label{intro_apriori}
  \| \nabla(\vecb{u} - \vecb{u}_h) \|^2_{L^2} \leq \! \inf_{\substack{\vecb{v}_h \in V_{0,h},\\ \vecb{u}_h
  = \vecb{v}_h \text{ on } \partial \Omega}} \! \| \nabla(\vecb{u} - \vecb{v}_h) \|^2_{L^2} + \frac{1}{\nu^2} \| \nabla p \|^2_{V_{0,h}^\star}.
\end{align}
This factor \(1/\nu\) causes a locking phenomenon. Indeed, for \(\nu \rightarrow 0\) or very complicated pressures, the pressure
contribution may dominate and lead to a very bad solution for the discrete velocity \(\vecb{u}_h\) \cite{Lin14,LM2015,wias:preprint:2177,LM2016}.

By a trick of \cite{Lin14} one can introduce a reconstruction operator \(\Pi\), that maps discretely divergence-free functions onto exactly divergence-free ones,
into the right-hand side and so transform any classical finite element method into a pressure-robust one. This replaces the pressure-dependent term in \eqref{intro_apriori} by a small consistency error of optimal order \cite{Lin14,BLMS15,LMT15,LM2016,2016arXiv160903701L}
and independent of \(1/\nu\).
Although this fixes the locking phenomenon and leads to huge gains in many numerical examples, efficient a posteriori error control for these methods
is an open problem. Efficient error estimators for the velocity error not only have to cope with the variational crime but also, and more importantly,
have to mimic the pressure-independence.

Standard residual-based a posteriori error estimators \(\eta\) usually have the form
\begin{align*}
  \| \nabla(\vecb{u} - \vecb{u}_h) \|_{L^2} \lesssim \eta := \eta_\text{vol} + \text{other terms}
\end{align*}
with a volume contribution \(\eta_\text{vol}\) and some other terms, like norms of the normal jumps of \(\vecb{u}_h\), data oscillations
or consistency errors.
In the standard residual-based error estimator for classical finite element methods \cite{MR2995179,MR3064266,MR993474,MR3425376}
the volume contribution takes the form (for any \(q \in H^1(\Omega)\) and piecewise Laplacian \(\Delta_\mathcal{T}\))
\begin{align}\label{intro_volume1}
   \eta_\text{vol} & = \nu^{-1} \| \nabla q \|_{V_{0,h}^\star} + \nu^{-1} \| h_\mathcal{T}(\vecb{f} - \nabla q + \nu \Delta_\mathcal{T} \vecb{u}_h) \|_{L^2}\\
   & \lesssim \| \nabla (\vecb{u} - \vecb{u}_h) \|_{L^2} + \nu^{-1} \left(\| p - q \|_{L^2} + \min_{q_h \in Q_h} \| q - q_h \|_{L^2} + \mathrm{osc_k}(\vecb{f} - \nabla q,\mathcal{T}) \right).\nonumber
\end{align}
The inequality above states efficiency, i.e. beeing also a lower bound of the real error,
and its dependence on the choice of \(q\).
Note, that for \(q \in Q_h\) the terms \(\| \nabla q \|_{V_{0,h}^\star} \leq \min_{q_h \in Q_h} \| q - q_h \|_{L^2} = 0\) vanish, but \(\| p - q \|_{L^2}\) remains, whereas
for \(q = p\) the term \(\| p - q \|_{L^2}\) vanishes, but the other two remain.
If the velocity error is at best as good as the error in the pressure (scaled by \(1/\nu\)), as it is the case for classical pressure-inrobust methods,
this estimate is fine (e.g.\ for \(q\) chosen as an \(H^1\)-interpolation of \(p_h\)).
As a result classical a posteriori error estimates, see e.g.~\cite{MR2995179,MR3064266,MR993474,MR3425376}, often perform the error analysis in a norm that
combines the velocity error and the pressure error. A pressure-robust method, however, allows for a decoupled error analysis of velocity error and pressure error
and so gives more control over both.

For a pressure-robust finite element method, the term \eqref{intro_volume1} can be replaced by
\begin{align}\label{intro_volume2}
  \eta_\text{vol} &= \nu^{-1} \| h_\mathcal{T}(\vecb{f} - \nabla q + \nu \Delta_\mathcal{T} \vecb{u}_h) \|_{L^2} \\
   &\lesssim \| \nabla (\vecb{u} - \vecb{u}_h) \|_{L^2} + \nu^{-1} \left(\| p - q \|_{L^2} + \mathrm{osc_k}(\vecb{f} - \nabla q,\mathcal{T}) \right).  \nonumber
\end{align}
Here, the choice \(q = p\) leads to a pressure-independent efficient estimate. However, this
cannot be considered a posteriori, since \(p\) is unknown. Hence, an efficient error estimator of this form for pressure-robust methods
hinges on a good approximation of \(q \approx p\) as already investigated in \cite{MR1471083,MR3366087}.

The main result of this paper concerns a different approach to estimate the velocity error that yields an estimator with the volume contribution
\begin{align}\label{intro_volume3}
   \eta_\text{curl} &= \nu^{-1} \| h_\mathcal{T}^2 \mathrm{curl}_\mathcal{T} (\vecb{f} + \nu \Delta_\mathcal{T} \vecb{u}_h) \|_{L^2} \\
   &\lesssim \| \nabla (\vecb{u} - \vecb{u}_h) \|_{L^2} + \nu^{-1} \mathrm{osc_k}( h_\mathcal{T} \mathrm{curl}_\mathcal{T}(\vecb{f} + \nu \Delta_\mathcal{T} \vecb{u}_h),\mathcal{T}). \nonumber
\end{align}
The advantage of \(\eta_\text{curl}\) over \(\eta_\text{vol}\) is that the \(\mathrm{curl}\) operator automatically cancels any \(\nabla q\) from the Helmholtz decomposition of \(\vecb{f} + \nu \Delta_\mathcal{T} \vecb{u}_h\)
and therefore no approximation of \(p\) as in \eqref{intro_volume2} is needed. Also note, that \(\eta_\text{curl}\) is similar to the volume contribution of a residual-based error
estimator for the Navier--Stokes equations in streamline and vorticity formulation \cite{MR1645033}.
However, the error estimator with this volume contribution is valid for any pressure-robust
finite element method like the Scott--Vogelius finite element method
\cite{scott:vogelius:conforming,Zhang05} or the novel family of pressure-robustly modified finite element methods of \cite{Lin14,BLMS15,LMT15,LM2016,2016arXiv160903701L} that allow for an interesting interplay between the Fortin interpolator \(I\)
and the reconstruction operator \(\Pi\) manifestated in the required assumption
\begin{align}\label{intro_assumption}
   \int_\Omega (1 - \Pi I) \vecb{v} \cdot \vecb{\theta} \dx
   \lesssim \| \nabla \vecb{v} \|_{L^2} \| h_\mathcal{T}^{2} \mathrm{curl} \vecb{\theta} \|_{L^2} \quad \text{for all } \vecb{\theta} \in H(\mathrm{curl},\Omega) \text{ and } \vecb{v} \in \vecb{V}_0.
\end{align}
We prove this assumption for certain popular finite element methods including the Taylor--Hood and mini finite element methods, and some elements with discontinuous pressure approximations. However, we only focus on the two-dimensional case, since the proofs
for the three-dimensional case are much more involved and therefore discussed in a future
publication.

The rest of the paper is structured as follows. Section~\ref{sec:stokes} introduces the Stokes equations and preliminaries as well as notation used throughout the paper.
Section~\ref{sec:prmethods} focuses on classical finite element methods and their recently developed pressure-robust siblings that are based on a suitable reconstruction operator.
Section~\ref{sec:standard_estimates} is concerned with standard residual-based error estimates for classical and pressure-robust finite element methods and
the efficiency of its contributions, in particular \eqref{intro_volume1} and \eqref{intro_volume2}, especially in the pressure-dominated regime.
Section~\ref{novelestimates} derives some novel a posteriori error bounds with the volume contribution \eqref{intro_volume3}
that are efficient and easy to evaluate for the pressure-robust finite element methods that satisfy Assumption \eqref{intro_assumption}.
In Section~\ref{sec:assumptionproof} this assumption is verified for many popular finite element methods and
their pressure-robust siblings.
Section~\ref{sec:numerics} studies numerical examples and employs the local contributions of the a posteriori error estimates as refinement indicators for
adaptive mesh refinement. The numerical examples verify the theory and show that the pressure-robust finite element methods converge with the optimal order
also in non-smooth examples.

\section{Model problem and preliminaries}\label{sec:stokes}
This section states our model problem and the needed notation.

\subsection{Stokes equations and Helmholtz projector}
The Stokes model problem seeks a vector-valued velocity field \(\vecb{u}\) and a scalar-valued pressure field \(p\) on a bounded Lipschitz domain \(\Omega \subset \mathbb{R}^2\) with Dirichlet data \(\vecb{u} = \vecb{u}_D\) along \(\partial \Omega\) and
\begin{align*}
  -\nu \Delta \vecb{u} + \nabla p & = \vecb{f} \quad \text{and} \quad \mathrm{div} \vecb{u} = 0 \quad \text{in } \Omega.
\end{align*}
The weak formulation characterises \(\vecb{u} \in H^1(\Omega)^2\) by \(\vecb{u} = \vecb{u}_D\) along \(\partial \Omega\) and
\begin{align*}
   \nu (\nabla \vecb{u}, \nabla \vecb{v}) - (p, \mathrm{div} \vecb{v}) & = (\vecb{f}, \vecb{v}) && \text{for all } \vecb{v} \in V := H^1_0(\Omega)^2,\\
				      (q, \mathrm{div} \vecb{u}) & = 0 && \text{for all } q \in Q := L^2_0(\Omega).
\end{align*}
In the set of divergence-free functions \(\vecb{V}_0 := \lbrace \vecb{v} \in V\, \lvert \, \mathrm{div} \vecb{v} = 0\rbrace\), \(\vecb{u}\)
satisfies
\begin{align*}
  \nu (\nabla \vecb{u}, \nabla \vecb{v}) = (\vecb{f}, \vecb{v}) \qquad \text{for all } \vecb{v} \in \vecb{V}_0.
\end{align*}
The Helmholtz decomposition decomposes every vector field into
\begin{align*}
  \vecb{f} = \nabla \alpha + \beta =: \nabla \alpha + \mathbb{P} \vecb{f}
\end{align*}
with \(\alpha \in H^1(\Omega) / \mathbb{R}\) and
\(\beta  =: \mathbb{P} \vecb{f} \in L^2_\sigma(\Omega) := \lbrace \vecb{w} \in H(\mathrm{div},\Omega) \, \vert \, \mathrm{div} \vecb{w} = 0, \vecb{w} \cdot \vecb{n} = 0 \text{ along } \partial D\rbrace\)
\cite{GR86}.
Note in particular, that the continuous Helmholtz projector satisfies \(\mathbb{P}(\nabla q) = 0\) for all \(q \in H^1(\Omega)\) which implies
\begin{align*}
  \nu (\nabla \vecb{u}, \nabla \vecb{v}) = (\mathbb{P} \vecb{f}, \vecb{v}) \qquad \text{for all } \vecb{v} \in \vecb{V}_0,
\end{align*}
hence \(\vecb{u}\) is steered only by the Helmholtz projector \(\mathbb{P} \vecb{f}\) of the right-hand side.

\subsection{Notation}
The set \(\mathcal{T}\) denotes a regular triangulation of \(\Omega\)
into two dimensional simplices with edges \(\mathcal{E}\) and nodes \(\mathcal{N}\).
The three edges of a simplex \(T \in \mathcal{T}\) are denoted by \(\mathcal{E}(T)\). Similarly,
\(\mathcal{N}(T)\) consists of the three nodes that belong to \(T \in \mathcal{T}\),
\(\mathcal{N}(E)\) consists of the two nodes that belong to \(E \in \mathcal{E}\)
and \(\mathcal{T}(z)\) for a vertex \(z \in \mathcal{N}\)
consists of all cells \(T \in \mathcal{T}\) with \(z \in \mathcal{N}(T)\). Finally we define $\mathcal{E}^\circ$ as the set of all inner. 

As usual \(L^2(\Omega)\), \(H^1(\Omega)\), \(H(\mathrm{div},\Omega)\) and \(H(\mathrm{curl},\Omega)\) denote the Sobolev spaces and
\(L^2(\Omega)^2\), \(H^1(\Omega)^2\) denote their vector-valued versions.
Moreover, several discrete function spaces are used throughout the paper. The set \(P_k(T)\)
denotes scalar-valued polynomials up to order \(k\) that live on the simplex \(T \in \mathcal{T}\)
and generate the global piecewise polynomials of order \(k\), i.e.
\begin{align*}
  P_k(\mathcal{T}) := \lbrace q_h \in L^2(\Omega) \, | \, \forall T \in \mathcal{T} : v_h|_T \in P_k(T) \rbrace.
\end{align*}
The function \(\pi_{P_k(\omega)}\) denotes the $L^2$ best approximation into \(P_k(\omega)\)
for any subdomain \(\omega \subset \Omega\).
For approximation of functions in $H(\mathrm{div},\Omega)$ we use the set of Brezzi-Douglas-Marini functions of order \(k \geq 1\) denoted by $BDM_k(\mathcal{T}):= P_k(\mathcal{T})^2 \cap H(\mathrm{div},\Omega)$
and the subset of Raviart-Thomas functions of order $k \ge 0$ denoted by $RT_{k}(\mathcal{T})$, see \cite{RTelements}.
The functions \(I_{RT_k}\) and \(I_{BDM_k}\) denotes the standard interpolator into \(RT_k(\mathcal{T})\) and \(BDM_k(\mathcal{T})\), respectively,
see e.g.\ \cite{MR3097958}. We are also using lowest order N\'ed\'elec (type I) functions $\mathcal{N}_{0}(\mathcal{T})$ defined as the 90 degree rotated lowest order Raviart-Thomas functions with the corresponding interpolator \(I_{\mathcal{N}_{0}}\), see \cite{MR592160}.

The diameter of a simplex \(T \in \mathcal{T}\) is denoted by \(h_T\) and \(h_\mathcal{T} \in P_0(\mathcal{T})\)
is the local mesh width function, i.e. \(h_\mathcal{T}|_T := h_T\) for all \(T \in \mathcal{T}\). Similarly, \(h_E\)
denotes the diameter of the side \(E \in \mathcal{E}\).
At some point certain bubble functions are used. The cell bubble function on a cell \(T \in \mathcal{T}\) is defined
by \(b_T = \prod_{z \in \mathcal{N}} \varphi_z\) where \(\varphi_z\) is the nodal basis function of the node \(z \in \mathcal{N}\),
i.e.\ \(\varphi_z(z) = 1\) and \(\varphi_z(y) = 0\) for \(y \in \mathcal{N} \setminus \lbrace z \rbrace\). Similarly, the
face bubble \(b_E\) for some side \(E \in \mathcal{E}\) is defined by \(b_E = \prod_{z \in \mathcal{E}} \varphi_z\).
The vector \(\vecb{n}_E\) denotes the unit normal vector of the side \(E \in \mathcal{E}\) with arbitrary but fixed orientation, such that
the normal jump \([ \vecb{v} \cdot \vecb{n}]\) of some function \(\vecb{v}\) has a well-defined sign. The vector \(\vecb{\tau}_E\) denotes a unit tangential vector of \(E\).

\section{Classical and pressure-robust finite element methods}\label{sec:prmethods}
This section recalls classical (usually not presssure-robust) inf-sup stable finite element methods and
a pressure-robust modification of these methods.

\subsection{Classical inf-sup stable finite element methods}
Classical inf-sup stable finite element methods choose ansatz spaces \(V_h \subseteq V = H^1_0(\Omega)^2\) and \(Q_h \subseteq Q = L^2_0(\Omega)\) with the inf-sup property
\begin{align}\label{eqn:infsupconstant}
0 < c_0 := \inf_{q_h \in Q_h \setminus \lbrace 0 \rbrace} \sup_{\vecb{v}_h \in V_h \setminus \lbrace \vecb{0} \rbrace}
\frac{\int_\Omega q_h \mathrm{div} \vecb{v}_h \dx}{\| \nabla \vecb{v}_h \|_{0} \| q_h \|_{L^2}}.
\end{align}
This guarantees surjectivity of the discrete divergence operator
\begin{align*}
  \mathrm{div}_h \vecb{v}_h = \Pi_{Q_h} (\mathrm{div} \vecb{v}_h) := \argmin_{q_h \in Q_h} \| \mathrm{div} \vecb{v}_h - q_h \|_{L^2},
\end{align*}
but also leads to the set of only discretely divergence-free testfunctions
\begin{align*}
  V_{0,h} = \lbrace \vecb{v}_h \in V_h \, \vert \, \mathrm{div}_h \vecb{v}_h = 0 \rbrace ,
\end{align*}
that in general is not a subset of the really divergence-free functions \(\vecb{V}_0\).
Table~\ref{tab:FEMtable} lists some classical finite element methods that are inf-sup stable and are considered in this paper. Besides
the Scott-Vogelius finite element method (on a barycentric refined mesh \(\mathrm{bary}(\mathcal{T})\) to ensure the inf-sup stability \cite{scott:vogelius:conforming,Zhang05}),
all of them are not divergence-free. The space \(P_k^+(\mathcal{T})\) in case of the P2-bubble \cite{crouzeix:raviart:1973} or the mini finite element methods \cite{ABF84}
indicates that the \(P_k(\mathcal{T})\) space is enriched with
the standard cell bubbles \(b_T\) for all \(T \in \mathcal{T}\).
For the Bernardi--Raugel finite element method
normal-weighted face bubbles are added \cite{BernardiRaugel} defining the space $ P_1^\text{BR}(\mathcal{T}) := P_1(\mathcal{T})^2 \cup \lbrace b_E \vecb{n}_E : E \in \mathcal{E} \rbrace. $

\begin{table}
  \begin{center}
    \caption{\label{tab:FEMtable}List of classical finite element methods that are considered in this paper and their expected velocity gradient error convergence order \(k\).}
    \footnotesize
      \begin{tabular}{l@{~~} @{~~}l @{~~} @{~~}l @{~~} @{~~}l @{~~} }
  FEM name \& reference \& order & abbr. & $V_h$ & $Q_h$\\
  \hline
  Bernardi--Raugel FEM \cite{BernardiRaugel} (\(k=1\)) & $\mathrm{BR}$ & $P_1^\text{BR}(\mathcal{T}) \cap V$ & $P_{0}(\mathcal{T})$\\
  Mini FEM \cite{ABF84} (\(k = 1\)) & $\mathrm{MINI}$ & $P_1^+(\mathcal{T})^2 \cap V$ & $P_{1}(\mathcal{T}) \cap H^1(\Omega)$\\
  $P_{k+1} \times P_{k-1}$ FEM (\(k \geq 1\)) & $\mathrm{P2P0}$,... & $P_{k+1}(\mathcal{T})^2 \cap V$ & $P_{k-1}(\mathcal{T})$\\
  P2-bubble FEM \cite{crouzeix:raviart:1973} (\(k=2\)) & $\mathrm{P2B}$ & $P_2^+(\mathcal{T})^2 \cap V$ & $P_{1}(\mathcal{T})$\\
  Taylor--Hood FEM \cite{HT74} (\(k \geq 2\))  & $\mathrm{TH}_k$ & $P_k(\mathcal{T})^2 \cap V$ & $P_{k-1}(\mathcal{T}) \cap H^1(\Omega)$\\
  Scott-Vogelius FEM \cite{scott:vogelius:conforming,Zhang05} (k=2) & $\mathrm{SV}$ & $P_{2}(\mathrm{bary}(\mathcal{T}))^2 \cap V$ & $P_{1}(\mathrm{bary}(\mathcal{T}))$\\
  \hline
\end{tabular}
\end{center}
\end{table}

The relaxation of the divergence constraint leads to the usual best approximation error in the pressure ansatz space, i.e.
\begin{align}
  \| \nabla p \|_{V_{0,h}^\star}
  &:= \sup_{\vecb{v}_h \in V_{0,h} \setminus \lbrace \vecb{0} \rbrace} \frac{\int_\Omega p \mathrm{div} \vecb{v}_h \dx}{\| \nabla \vecb{v}_h \|_{L^2}} \label{eqn:divrelax_consistency_error}\\
 & = \sup_{\vecb{v}_h \in V_{0,h} \setminus \lbrace \vecb{0} \rbrace} \frac{\int_\Omega (p - q_h) \mathrm{div} \vecb{v}_h \dx}{\| \nabla \vecb{v}_h \|_{L^2}}
  \leq \min_{q_h \in Q_h} \| p - q_h \|_{L^2},\nonumber
\end{align}
and divergence-free methods are characterised by
\begin{align*}
  V_{0,h} \subseteq \vecb{V}_0 \ \Leftrightarrow \ \| \nabla p \|_{V_{0,h}^\star} = 0 \text{ for all } p \in L^2(\Omega).
\end{align*}

For completeness, we shortly prove the classical a priori error estimate in the following theorem for the discrete solution \(\vecb{u}_h \in \vecb{u}_{D,h} + V_h\) 
(where \(\vecb{u}_{D,h}\) is some suitable approximation of \(\vecb{u}_{D}\))
and \(p_h \in Q_h\) defined by
\begin{align}\label{eqn:classical_solution}
  \nu ( \nabla \vecb{u}_h, \nabla \vecb{v}_h) - (p_h,\mathrm{div} \vecb{v}_h) & = (\vecb{f},\vecb{v}_h) && \text{for all } \vecb{v}_h \in V_h,\\
				      (q_h,\mathrm{div} \vecb{u}_h) & = 0 && \text{for all } q_h \in Q_h,\nonumber
\end{align}
or, equivalently,
\begin{align*}
  \nu ( \nabla \vecb{u}_h, \nabla \vecb{v}_h) = (\vecb{f},\vecb{v}_h) \qquad \text{for all } \vecb{v}_h \in V_{0,h}.
\end{align*}

\begin{theorem}[A priori estimate for classical finite element methods]\label{thm:apriori_classical}
For the discrete velocity $\vecb{u}_h$ of \eqref{eqn:classical_solution}, it holds
\begin{align*}
  \| \nabla(\vecb{u} - \vecb{u}_h) \|^2_{L^2} \leq \! \inf_{\substack{\vecb{v}_h \in V_{0,h},\\ \vecb{u}_h = \vecb{v}_h \text{ on } \partial \Omega}} \! \| \nabla(\vecb{u} - \vecb{v}_h) \|^2_{L^2} + \frac{1}{\nu^2} \| \nabla p \|^2_{V_{0,h}^\star}.
\end{align*}
\end{theorem}
\begin{proof}
  The best approximation \(\vecb{w}_h \in V_{0,h}\) with boundary data \(\vecb{w}_h = \vecb{u}_h\) along \(\partial \Omega\) of \(\vecb{u}\) in the \(H^1\)-seminorm satisfies in particular
  the orthogonality \( ( \nabla (\vecb{u} - \vecb{w}_h), \nabla (\vecb{u}_h - \vecb{w}_h) ) = 0\)
  and therefore allows for the Pythagoras theorem
  \begin{align}
     \| \nabla(\vecb{u} - \vecb{u}_h) \|^2_{L^2}
     &= \| \nabla(\vecb{u} - \vecb{w}_h) \|^2_{L^2} + \| \nabla(\vecb{u}_h - \vecb{w}_h) \|^2_{L^2} \nonumber\\
    &= \!\!\! \inf_{\substack{\vecb{v}_h \in V_{0,h},\\ \vecb{u}_h = \vecb{v}_h \text{ on } \partial \Omega}} \!\!\! \| \nabla(\vecb{u} - \vecb{v}_h) \|^2_{L^2} + \| \nabla(\vecb{u}_h - \vecb{w}_h) \|^2_{L^2}.\label{eqn:Pythagoras_apriori}
  \end{align}
  The same orthogonality allows to estimate
  \begin{align*}
     \| \nabla(\vecb{u}_h - \vecb{w}_h) \|^2_{L^2}
     & = ( \nabla (\vecb{u} - \vecb{u}_h), \nabla (\vecb{u}_h - \vecb{w}_h) )\\
     & = \nu^{-1} ( p, \mathrm{div} (\vecb{u}_h - \vecb{w}_h))
      \leq \nu^{-1} \| \nabla p \|_{V_{0,h}^\star}  \| \nabla(\vecb{u}_h - \vecb{w}_h) \|_{L^2}.\ 
  \end{align*}
\end{proof}
The malicious influence of the pressure-dependent error and the factor \(1/\nu\) in front of it
for classical finite element methods that are not divergence-free
was demonstrated and observed in many benchmark examples, see e.g.\ \cite{LM2015,LM2016,wias:preprint:2177,2016arXiv160903701L}.

\subsection{Pressure-robust finite element methods}
A method is called pressure-robust if its discrete velocity is pressure-independent, i.e.\ if the a priori error estimate for the velocity error
is independent of the pressure.

The key feature behind pressure-robustness for the Stokes problem
is that the testfunctions in the right-hand side are diver\-gence-free.
This can be achieved e.g.\ by fully divergence-free finite element methods (like the Scott-Vogelius finite element method)
or, focused on in this paper, by the application of some reconstruction
operator \(\Pi\) in the right-hand side of the equation (and in further terms in case of the stationary and transient Navier--Stokes equations \cite{LM2016,WIASPreprint2368}).

Hence, the modified pressure-robust finite element method (of any classical pair of inf-sup stable spaces
\(V_h\) and \(Q_h\)) searches
\(\vecb{u}_h \in \vecb{u}_{D,h} + V_h\) and \(p_h \in Q_h\) with
\begin{align}\label{eqn:modified_solution}
  \nu ( \nabla \vecb{u}_h, \nabla \vecb{v}_h) - (p_h,\mathrm{div} \vecb{v}_h) & = (\vecb{f},\Pi \vecb{v}_h) = (\mathbb{P} \vecb{f},\Pi \vecb{v}_h) && \text{for all } \vecb{v}_h \in V_h,\\
				      (q_h,\mathrm{div} \vecb{u}_h) & = 0 && \text{for all } q_h \in Q_h.\nonumber
\end{align}
The operator \(\Pi\) maps discretely divergence-free functions onto exactly divergence-free ones, i.e.
\begin{align}\label{eqn:def_reconstoperator}
  \Pi : V_h \rightarrow H(\mathrm{div},\Omega) \quad \text{with} \quad \mathrm{div} (\Pi\vecb{v}_h) = 0 \quad \text{for all } \vecb{v}_h \in V_{0,h}.
\end{align}
This implicitly defines a modified discrete Helmholtz projector
\begin{align*}
  \mathbb{P}_h^\star \vecb{f} = \argmin_{\vecb{v}_h \in V_{0,h}} \| \vecb{f} - \Pi \vecb{v}_h \|_{L^2}
\end{align*}
with \(\mathbb{P}_h^\star(\nabla q) = 0\) for any \(q \in H^1(\Omega)\) or \(\| \nabla q \|^2_{(\Pi V_{0,h})^\star} = 0\) for all \(q \in L^2(\Omega)\) and so allows for a pressure-independent and locking-free a priori velocity error estimate.

\begin{theorem}[A priori estimate for pressure-robust finite element methods]\label{thm:apriori_probust}
For the solution \(\vecb{u}_h\) of \eqref{eqn:modified_solution} with a reconstruction operator \(\Pi\)
that satisfies \eqref{eqn:def_reconstoperator}, it holds
\begin{align*}
  \| \nabla(\vecb{u} - \vecb{u}_h) \|^2_{L^2} \leq \! \inf_{\substack{\vecb{v}_h \in V_{0,h},\\ \vecb{u}_h = \vecb{v}_h \text{ on } \partial \Omega}} \! \| \nabla(\vecb{u} - \vecb{v}_h) \|^2_{L^2} +  \| \Delta \vecb{u} \circ (1 - \Pi) \|^2_{V_{0,h}^\star}
\end{align*}
with the consistency error
\begin{align}\label{def:Pi_consistency_error}
  \| \Delta \vecb{u} \circ (1 - \Pi) \|^2_{V_{0,h}^\star} :=
  \sup_{\vecb{v}_h \in V_{0,h} \setminus \lbrace \vecb{0} \rbrace} \frac{\int_\Omega \Delta \vecb{u}
  \cdot (1 - \Pi)\vecb{v}_h \dx}{\| \nabla \vecb{v}_h \|_{L^2}}.
\end{align}
Note, that divergence-free methods (like the Scott-Vogelius finite element method)
allow for \(\Pi = 1\) and so attain the same estimate as Theorem~\ref{thm:apriori_classical}.
\end{theorem}
\begin{proof}
Similar to the proof of Theorem~\ref{thm:apriori_classical}, it remains to estimate
the second term on the right-hand side of \eqref{eqn:Pythagoras_apriori}.
  Using the orthogonality \( ( \nabla (\vecb{u} - \vecb{w}_h), \nabla (\vecb{u}_h - \vecb{w}_h) ) = 0\) we get  $\| \nabla(\vecb{u}_h - \vecb{w}_h) \|^2_{L^2} =
       ( \nabla (\vecb{u} - \vecb{u}_h), \nabla (\vecb{u}_h - \vecb{w}_h) )$. The insertion of \(\vecb{f} = - \nu \Delta \vecb{u} + \nabla p\)
  and \(\int_\Omega \nabla p \cdot \Pi (\vecb{u}_h - \vecb{w}_h) = 0\) (thanks to \eqref{eqn:def_reconstoperator}) then further shows
  \begin{align*}
       ( \nabla (\vecb{u} - \vecb{u}_h), \nabla (\vecb{w}_h - \vecb{u}_h) )
     & = ( \Delta \vecb{u} , \vecb{u}_h - \vecb{w}_h ) + \frac{1}{\nu} ( \vecb{f} , \Pi (\vecb{u}_h - \vecb{w}_h) ) \\
     & = ( \Delta \vecb{u} , \vecb{u}_h - \vecb{w}_h ) + ( \Delta \vecb{u} , \Pi (\vecb{u}_h - \vecb{w}_h) ) \\
     &\leq \| \Delta \vecb{u} \circ (1 - \Pi) \|_{V_{0,h}^\star} \| \nabla(\vecb{u}_h - \vecb{w}_h) \|_{L^2}.
  \end{align*}
  This concludes the proof.
\end{proof}

To gain optimal convergence behavior of \eqref{def:Pi_consistency_error}, the reconstruction operator additionally has to satisfy another important
property that concerns the consistency error of the modified method.
For a finite element method with optimal \(H^1\)-velocity convergence order \(k\) 
and pressure \(L^2\)-convergence order \(q\) we require, for all \(\vecb{v}_h \in V_{0,h}\),
\begin{align}
  (\vecb{g}, (1 - \Pi) \vecb{v}_h) \label{eq:orthogonality_on_polynomials}
  & \lesssim \| h_\mathcal{T}^{q+1} D^{q-1}\vecb{g} \|_{L^2(\Omega)} \| \nabla \vecb{v}_h \|_{L^2} \quad && \text{for any } \vecb{g} \in H^{q-1}(\Omega)^2.
\end{align}
In particular, for \(\Delta \vecb{u} \in H^{q-1}(\Omega)^2\), this property directly implies
\begin{align}\label{eq:consistency_error}
  \|\Delta \vecb{u} \circ (1-\Pi)\|_{V_{0,h}^\star}
  \lesssim \| h_\mathcal{T}^{q+1} D^{q-1} \Delta \vecb{u} \|_{L^2(\Omega)}
\end{align}
and so ensures that the modified method still converges with the optimal order. 

To be more precise, we require that the reconstruction operator satisfies some local splitting and orthogonality property that can
be formulated by
  \begin{align}\label{eqn:reconstop_local_orthogonalities}
     (1 - \Pi) \vecb{v}_h & = \sum_{K \in \mathcal{K}} \vecb{\sigma}_K|_{K} \ \text{with} \ \| \vecb{\sigma}_K \|_{L^2(K)}
     \lesssim h_K \| \nabla \vecb{v}_h \|_{L^2(K)} \ \text{and}\\
     \int_{K} \vecb{\sigma}_K \cdot \vecb{g}_h \dx & = 0 \quad \text{for all } \vecb{g}_h \in P_{q-1}(K),\nonumber
  \end{align}
with $h_K := \mathrm{diam}(K)$.
Reconstruction operators \(\Pi\) with the properties \eqref{eqn:def_reconstoperator}-\eqref{eq:orthogonality_on_polynomials}
were already successfully designed for finite element methods with discontinuous pressure spaces, like the nonconforming Crouzeix-Raviart
finite element method \cite{Lin14,BLMS15}, or the Bernardi--Raugel \cite{LM2016} and \(P^2\)-bubble finite element methods \cite{LMT15,LM2016}.
In all these cases \(\Pi\) can be chosen as standard BDM interpolators with
elementwise-orthogonality with resepect to \(\mathcal{K} = \mathcal{T}\).
Recently, also for Taylor--Hood and mini finite element methods (with $k=q$) of arbitrary order such an operator was found \cite{2016arXiv160903701L}. For these
vertex-based constructions Property \eqref{eqn:reconstop_local_orthogonalities} holds with \(\mathcal{K} = \lbrace \omega_z : z \in \mathcal{N} \rbrace\).
Table~\ref{tab:RECtable} summarizes
suitable reconstruction operators, that satisfy the needed properties, for the methods from Table~\ref{tab:FEMtable}.

\begin{table}
  \begin{center}
    \caption{\label{tab:RECtable}Suitable reconstruction operators $\Pi$ for the classical FEMs of Table~\ref{tab:FEMtable}.}
    \footnotesize
            \begin{tabular}{l@{~~} @{~~}l @{~~} @{~~}l @{~~} }
  FEM name & abbr. & $\Pi$ \& reference\\
  \hline
  Bernardi--Raugel FEM & $\mathrm{BR}$ & $I_{BDM_1}$, see \cite{LM2016}\\
  Mini FEM & $\mathrm{MINI}$ & see \cite{2016arXiv160903701L} \\
  $P_{k+1} \times P_{k-1}$ FEM (\(k \geq 1\))& $\mathrm{P2P0}$, $\mathrm{P3P1}$, ... & $I_{BDM_k}$ \\
  P2-bubble FEM & $\mathrm{P2B}$ &  $I_{BDM_2}$, see \cite{LMT15,LM2016}\\
  Taylor--Hood FEM (\(k \geq 1\)) & $\mathrm{TH}_k$ & see \cite{2016arXiv160903701L}\\
  Scott-Vogelius FEM & $\mathrm{SV}$ & $1$ (identity)\\
  \hline
\end{tabular}
  \end{center}
\end{table}

\section{(Limits of) Standard a posteriori residual-based error bounds}\label{sec:standard_estimates}
This section states and proves a posteriori error bounds for the classical and the pressure-robust finite element methods
by classical means. The resulting bounds reflect the pressure-robustness but are, in case of a
pressure-robust finite element method, rather unhandy as their efficiency relies on a good approximation of \(\mathbb{P} \vecb{f}\).
To stress this observation, the analysis is performed in some detail.

First, we define the residual for the Stokes equations by
\begin{align*}
   r(\vecb{v}) &:= \int_\Omega \vecb{f} \cdot \vecb{v} \dx - \int_\Omega \nu \nabla \vecb{u}_h : \nabla \vecb{v}\dx \quad \text{for all } \vecb{v} \in \vecb{V}_0.
\end{align*}
The dual norm of the residual \(r\) with respect to \(\vecb{V}_0\) defined by
\begin{align*}
  \| r \|_{\vecb{V}_0^\star} := \sup_{\vecb{v}_h \in V_{0} \setminus \lbrace \vecb{0} \rbrace} \frac{r(\vecb{v})}{\| \nabla \vecb{v} \|_{L^2}}
\end{align*}
enters the generalised error bound as the central object of a posteriori error estimation.
The error analysis also assumes the existence of a Fortin interpolation operator \(I\)
that maps from $\vecb{V}_0$ to $V_{0,h}$ and has first-order approximation properties and is \(H^1\)-stable, i.e,
for all \(\vecb{v} \in \vecb{V}_0\), it holds
\begin{align}
  \| (1- I) \vecb{v} \|_{L^2(T)} & \lesssim h_T \| \nabla \vecb{v} \|_{L^2(\omega_T)} \quad \text{for all } T\in \mathcal{T},\label{eqn:Fortinprops1} \\
  \| \nabla I\vecb{v} \|_{L^2} & \lesssim \| \nabla \vecb{v} \|_{L^2} \label{eqn:Fortinprops2} .
\end{align}
For many classical finite element methods such an operator can be found in \cite{MR3097958}.
For its existence and design in the Taylor--Hood case we refer to \cite{Mardal2013,MR3272546}.
Some more details are given in Section~\ref{sec:assumptionproof} below.

The following theorem establishes a general estimate similar to \cite[Theorem 5.1]{MR2995179}
and can be extended to nonconforming methods in a similar fashion. However, our focus is
on the consistency errors \eqref{eqn:divrelax_consistency_error} and \eqref{def:Pi_consistency_error} and the dependency on \(\nu\).

\begin{theorem}\label{thm:errorbounds}
The following velocity error estimates hold:
\begin{itemize}
 \item[(a)]
  In general, the \(L^2\) gradient error can be estimated by
\begin{align*}
\| \nabla(\vecb{u} - \vecb{u} _h) \|_{L^2}^2 \leq \nu^{-2} \| r \|_{\vecb{V}_0^\star}^2 + 1/c_0^2 \| \mathrm{div} \vecb{u}_h \|_{L^2}^2.
\end{align*}
 \item[(b)]
  For the discrete solution \(\vecb{u}_h\) of the modified method \eqref{eqn:modified_solution}
  (or of the classical method \eqref{eqn:classical_solution} with \(\Pi = 1\)),
  the dual norm of the residual \(r\) can be bounded by
  \begin{align*}
    \| r \|_{\vecb{V}_0^\star} \lesssim \eta_\text{class}(\sigma,q) := 
    \eta_\text{vol}(\sigma,q)
    + \eta_\text{avg}(\sigma)
    + \eta_\text{jump}(\sigma)
    + \eta_\text{cons,1}(\sigma)
    + \eta_\text{cons,2}(q)
  \end{align*}
  for arbitrary \(q \in H^1(\Omega)\) and \(\sigma \in H^1(\mathcal{T})^{2 \times 2}\). The
  subterms read
  \begin{align*}
     \eta_\text{vol}(\sigma,q) & := \| h_{\mathcal{T}}(\vecb{f} - \nabla q + \nu \mathrm{div}_h(\sigma)) \|_{L^2}\\
     \eta_\text{avg}(\sigma) & := \nu \| \nabla \vecb{u}_h - \sigma \|_{L^2}\\
     \eta_\text{jump}(\sigma) & := \| h_\mathcal{E}^{1/2} [\nu \sigma \vecb{n}_E] \|_{L^2(\bigcup \mathcal{E}^\circ)}\\
     \eta_\text{cons,1}(\sigma) & := \| \nu \mathrm{div}_h (\sigma) \circ (1 - \Pi) \|_{V_{0,h}^\star}\\
     \eta_\text{cons,2}(q) & := \| \nabla q \|_{(\Pi V_{0,h})^\star},\hspace{1cm} 
  \end{align*}
  Note that \(q\) acts as a conforming approximation of the pressure \(p\) and \(\sigma\) acts as an
  approximation of \(\nabla \vecb{u}\)  (in particular \(\sigma = \nabla \vecb{u}_h\) is allowed).
\end{itemize}
\end{theorem}
\begin{proof}
The proof of (a) can be found in \cite{MR2164088,ccmerdon:nonconforming2} and is based on the decomposition
 $  \nu \nabla (\vecb{u} - \vecb{u} _h) = \nu \nabla \vecb{z} + y$
 into some \(\vecb{z} \in \vecb{V}_0\) and some remainder
 \begin{align*}
y \in Y := \left\lbrace y \in L^2(\Omega)^{d \times d} \, \vert \, \int_\Omega y : \nabla v \dx = 0 \text{ for all } v \in \vecb{V}_0 \right\rbrace.
 \end{align*}
The orthogonality relations between \(\vecb{z}\) and \(y\) lead to
\begin{align*}
\| \nu^{1/2} \nabla(\vecb{u} - \vecb{u} _h) \|_{L^2}^2 = \|\nu^{1/2} \nabla \vecb{z} \|_{L^2}^2 + \| \nu^{-1/2} y \|_{L^2}^2.
\end{align*}
Since
\begin{align*}
  \|\nu^{1/2} \nabla \vecb{z} \|_{L^2}^2 = \int_\Omega \nu \nabla(\vecb{u} - \vecb{u}_h) : \nabla \vecb{z} \dx = r(\vecb{z}) \leq \nu^{-1/2} \| r \|_{\vecb{V}_0^\star} \|\nu^{1/2} \nabla \vecb{z} \|_{L^2},
\end{align*}
one arrives at
$ \|\nu^{1/2} \nabla \vecb{z} \|_{L^2} \leq \nu^{-1/2} \| r \|_{\vecb{V}_0^\star}$. 
This is in fact an identity, since
\begin{align*}
r(\vecb{v}) = \int_\Omega \nu \nabla \vecb{z} : \nabla \vecb{v} \dx \leq \nu^{1/2} \|\nu^{1/2} \nabla \vecb{z} \|_{L^2} \|\nabla \vecb{v} \|_{L^2}
\quad \text{for any } \vecb{v} \in \vecb{V}_0.
\end{align*}
Furthermore, there exists some \(w \in L^2(\Omega)\) such that (see \cite{MR2164088})
\begin{align*}
  \| \nu^{-1/2} y \|_{L^2}^2
  & = \int_\Omega \nabla (\vecb{u} - \vecb{u} _h) : y \dx
  = \int_\Omega w \mathrm{div} (\vecb{u} - \vecb{u_h}) \dx \\
  & \leq \| w \|_{L^2} \| \mathrm{div} (\vecb{u} - \vecb{u_h})  \|_{L^2}
  \leq \nu^{1/2}/c_0 \| \nu^{-1/2} y \|_{L^2} \| \mathrm{div} \vecb{u}_h \|_{L^2}.
\end{align*}
Hence, \(\| \nu^{-1/2} y \|_{L^2} \leq \nu^{1/2}/c_0 \| \mathrm{div} \vecb{u}_h \|_{L^2}\).
This concludes the proof of (a) and it remains to prove (b).

Given any \(\vecb{v} \in \vecb{V}_0\),  subtraction of its Fortin interpolation \(I \vecb{v} \in V_{0,h}\)
and \eqref{eqn:modified_solution} lead to
\begin{align*}
  r(\vecb{v}) &= \int_\Omega \vecb{f} \cdot \vecb{v} \dx - \int_\Omega \nu \nabla \vecb{u}_h : \nabla \vecb{v} \dx \\
              & = \int_\Omega \vecb{f} \cdot (1 - \Pi I) \vecb{v} \dx - \int_\Omega \nu \nabla \vecb{u}_h : \nabla (1 - I) \vecb{v} \dx \nonumber \\
              & = \int_\Omega \vecb{f} \cdot (1 - \Pi I) \vecb{v} \dx
  - \int_\Omega \nu \sigma : \nabla (1- I)\vecb{v} \dx - \int_\Omega \nu (\nabla \vecb{u}_h - \sigma): \nabla (1- I)\vecb{v}\dx \nonumber\\
	      & = \int_\Omega (\vecb{f} -\nabla q + \nu \mathrm{div}_h \sigma) \cdot (1 - \Pi I) \vecb{v} \dx + \sum_T \int_{\partial T} (\nu \sigma \vecb{n}) \cdot (1 - I) \vecb{v} \ds \nonumber\\
	      &  - \int_\Omega \nu (\nabla \vecb{u}_h - \sigma): \nabla (1- I)\vecb{v} \dx + \int_\Omega \nu \mathrm{div}_h \sigma \cdot (1 - \Pi) I \vecb{v} \dx + \int_\Omega \nabla q \cdot \Pi I \vecb{v} \dx. \nonumber   
\end{align*}
In the last step it was used that \(\int \nabla q \cdot \vecb{v} \dx = 0\) for any \(q \in H^1(\Omega)\), since \(\vecb{v} \in \vecb{V}_0\) is divergence-free.
The third integral is estimated by a Cauchy inequality and the \(H^1\)-stability of \(I\).
The last two integrals are estimated by discrete dual norms and the \(H^1\)-stability of \(I\).
Properties \eqref{eqn:Fortinprops1}-\eqref{eqn:Fortinprops2} of \(I\) and \eqref{eqn:reconstop_local_orthogonalities} of \(\Pi\) show
\begin{align*}
  \| h_\mathcal{T}^{-1} (1 - \Pi I) \vecb{v} \|_{L^2(T)}
  & \leq \| h_\mathcal{T}^{-1} (1 - I) \vecb{v} \|_{L^2(T)} + \| h_\mathcal{T}^{-1} (1 - \Pi) I \vecb{v} \|_{L^2(T)}\\
  & \lesssim \| \nabla \vecb{v} \|_{L^2(\omega_T)} + \|h_\mathcal{T} \nabla (I \vecb{v}) \|_{L^2(\omega_T)}
  \lesssim \|\nabla \vecb{v} \|_{L^2(\omega_T)}
\end{align*}
and hence together with some Cauchy inequalities
\begin{align*}
 \int_\Omega (\vecb{f} -\nabla q + \nu \Delta_\mathcal{T} \vecb{u}_h) &\cdot (1 - \Pi I) \vecb{v} \dx \\
 & \leq \sum_{T \in \mathcal{T}} \| h_\mathcal{T}(\vecb{f} -\nabla q + \nu \Delta_\mathcal{T} \vecb{u}_h) \|_{L^2(T)} \| h_\mathcal{T}^{-1} (1 - \Pi I) \vecb{v} \|_{L^2(T)}\\
 & \lesssim \eta_\text{vol}(\sigma,q) \left(\sum_{T \in \mathcal{T}} \|\nabla \vecb{v} \|_{L^2(\omega_T)}^2\right)^{1/2} \lesssim \eta_\text{vol}(\sigma,q) \|\nabla \vecb{v} \|_{L^2}.
\end{align*}
Similar arguments hold for the edge-based integral using a trace inequality and Properties \eqref{eqn:Fortinprops1}-\eqref{eqn:Fortinprops2}, i.e.
\begin{align*}\sum_{E \in \mathcal{E}^\circ}
  \int_E [\nu \sigma\vecb{n}] \cdot (\vecb{v} - I \vecb{v}) \ds
  & \leq \sum_{E \in \mathcal{E}^\circ} \| [\nu \sigma\vecb{n}] \|_{L^2(E)} \|\vecb{v} - I \vecb{v}\|_{L^2(E)}\\
  & \leq \sum_{E \in \mathcal{E}^\circ} h_E^{1/2} \| [\nu \sigma\vecb{n}] \|_{L^2(E)} \|\nabla \vecb{v}\|_{L^2(\omega_E)}\\
  & \leq \|h_\mathcal{E}^{1/2}  [\nu \sigma\vecb{n}] \|_{L^2( \mathcal{E}^\circ)} \| \nabla \vecb{v} \|_{L^2} = \eta_\text{jump}(\sigma) \| \nabla \vecb{v} \|_{L^2}.
\end{align*}
This concludes the proof of (b).
\end{proof}

\begin{remark}
Some remarks are in order:
\begin{itemize}[fullwidth]
 \item
The existence of \(w\) in the last part of the proof of (a) needs \(\vecb{u} - \vecb{u}_h \in H^1_0(\Omega)^2\).
In case of inhomogeneous Dirichlet boundary data or nonconforming discretisations \(\vecb{u}_h \notin H^1(\Omega)^2\), one can introduce a function \(\vecb{w} \in H^1(\Omega)\)
(e.g.\ the harmonic extension of the boundary data error \(\vecb{u}_D - \vecb{u}_{D,h}\) \cite{MR2101782}
plus some \(H^1\)-conforming boundary-data preserving interpolation of \(\vecb{u}_h\) \cite{MR2164088,MR2995179,ccmerdon:nonconforming2})
with \(\vecb{w} = \vecb{u}_D\) along \(\partial \Omega\) and attains \(\vecb{u} - \vecb{w} \in H^1_0(\omega)\). Then, a modified estimation of the second term yields
\begin{align*}
   \| \nu^{-1/2} y \|_{L^2} \leq \nu^{1/2}/c_0 \| \mathrm{div} \vecb{w} \|_{L^2} + \nu^{1/2} \| \nabla_h ( \vecb{u}_h - \vecb{w}) \|_{L^2}.
\end{align*}
 \item 
 The term \(\eta_\text{cons,1}(\sigma) = \| \nu \Delta_\mathcal{T} (\mathrm{div} \sigma) \circ (1 - \Pi) \|_{V_{0,h}^\star}\) only appears for \(\Pi \neq 1\)
 as in the novel pressure-robust methods and equals the consistency error 
 \eqref{def:Pi_consistency_error} for \(\sigma = \nabla \vecb{u}_h\).
 \item Recall that \(\eta_\text{cons,2}(q) = 0\) if \(\Pi\) satisfies \eqref{eqn:def_reconstoperator} or if \(q \in Q_h\) and \(\Pi = 1\).
\end{itemize}
\end{remark}

The following theorem studies the efficiency of the contributions of the standard residual error estimators from Theorem~\ref{thm:errorbounds}
for the explicit choice \(\sigma = \nabla \vecb{u}_h\).

\begin{theorem}[Efficiency for \(\sigma = \nabla \vecb{u}_h\)]\label{thm:efficiency}
 For \(\sigma = \nabla \vecb{u}_h\) all terms of the residual-based error estimator of Theorem~\ref{thm:errorbounds} are efficient
 possibly up to data oscillations
\begin{align*}
  \mathrm{osc}_k(\bullet,\mathcal{T})^2 := \sum_{T \in \mathcal{T}} h_T^2 \| (1 - \pi_{P_k(T)}) \bullet \|_{L^2(T)}^2
\end{align*}
and up to pressure contributions (either from the lack of pressure-robustness or from the quality of the approximation of \(p\) by \(q\))
in the following sense.
 \begin{itemize}
  \item[(a)] For the divergence term it holds \(\| \mathrm{div} \vecb{u}_h \|_{L^2} \leq \| \nabla (\vecb{u} - \vecb{u}_h) \|_{L^2}.\)
  \item[(b)] For the volume term \(\eta_\text{vol}(q,\nabla \vecb{u}_h)\), it holds
    \begin{align*}
      \nu^{-1} \| h_T(\vecb{f} - \nabla q + \nu \Delta_\mathcal{T} \vecb{u}_h) \|_{L^2}
       &\lesssim \| \nabla (\vecb{u} - \vecb{u}_h) \|_{L^2} \\
       & \quad + \nu^{-1} \left(\| p - q \|_{L^2} + \mathrm{osc_k}(\vecb{f} - \nabla q,\mathcal{T})\right).
    \end{align*}  
  \item[(c)] For the jump term \(\eta_\text{jump}(\nabla \vecb{u}_h)\), it holds
    \begin{align*}
      \nu^{-1} \| h_\mathcal{E}^{1/2} [\nu \nabla \vecb{u}_h \vecb{n}_E] \|_{L^2(\bigcup \mathcal{E^\circ})}  \lesssim \| \nabla (\vecb{u} - \vecb{u}_h) \|_{L^2} + \nu^{-1} \mathrm{osc_k}(\vecb{f} - \nabla p,\mathcal{T}).
    \end{align*}
  \item[(d)] If \(\Pi\) satisfies \eqref{eqn:reconstop_local_orthogonalities}, the consistency error \(\eta_\text{cons,1}(\nabla \vecb{u})\) is efficient in the sense
  \begin{align*}
\| \Delta_\mathcal{T} \vecb{u}_h \circ (1 - \Pi) \|_{V_{0,h}^\star} 
\lesssim \| \nabla (&\vecb{u} - \vecb{u}_h) \|_{L^2} \\ &+ \nu^{-1} \left(\mathrm{osc}_{k}(\vecb{f} - \nabla p,\mathcal{T}) + \mathrm{osc}_{q-1}(\vecb{f} - \nabla p,\mathcal{K}) \right)
\end{align*}  
  \item[(e)] For the consistency error \(\eta_\text{cons,2}(q)\), it holds
    \begin{align*}
       \| \nabla q \|_{(\Pi V_{0,h})^\star} & \leq
       \begin{cases}
          0 & \text{if } \Pi \text{ satisfies } \eqref{eqn:def_reconstoperator},\\
          \min_{q_h \in Q_h} \| q - q_h \|_{L^2} & \text{if } \Pi = 1 \text{ without } \eqref{eqn:def_reconstoperator}.
       \end{cases}
    \end{align*}
 \end{itemize}
\end{theorem}
\begin{proof}
The proof of (a) simply uses \(\mathrm{div} \vecb{u} = 0\) to estimate
  \begin{align*}
  \| \mathrm{div} \vecb{u}_h \|_{L^2} =  \| \mathrm{div} (\vecb{u} - \vecb{u}_h) \|_{L^2} \leq \| \nabla (\vecb{u} - \vecb{u}_h) \|_{L^2}.
\end{align*}
The last inequality follows from the identity \(\| \nabla \vecb{v} \|^2 = \| \mathrm{curl} \vecb{v} \|^2 + \| \mathrm{div} \vecb{v} \|^2\) for any \(\vecb{v} \in H^1_0(\Omega)^2\),
see e.g.\ \cite[Remark 2.6]{MR1626990}.

The proof of (b) and (c) is standard and employs the bubble-technique of Verf\"urth, see e.g. \cite{MR993474,MR1213837} or into the proof of Theorem~\ref{thm:efficiency_new_estimate} below.

  To show (d), observe that Property \eqref{eqn:reconstop_local_orthogonalities} leads to
  \begin{align*}
    &\int_\Omega \nu \Delta_\mathcal{T} \vecb{u}_h \cdot (1 - \Pi) \vecb{v}_h \dx = \sum_{K \in \mathcal{K}} \int_{K} \nu \Delta_\mathcal{T} \vecb{u}_h \cdot \sigma_K \dx\\
     & = \sum_{K \in \mathcal{K}} \int_{K}  (\vecb{f} - \nabla p + \nu \Delta_\mathcal{T} \vecb{u}_h) \cdot \sigma_K \dx - \int_{K} (1 - \pi_{P_{q-1}(K)})(\vecb{f} - \nabla p )\cdot \sigma_K \dx\\
     & \lesssim \sum_{K \in \mathcal{K}} h_K \left( \| \vecb{f} - \nabla p+ \nu \Delta_\mathcal{T} \vecb{u}_h\|_{L^2(K)} \right. \\ &\qquad \qquad \qquad \qquad\qquad \qquad \left. + \|(1 - \pi_{P_{q-1}(K)}) (\vecb{f} - \nabla p) \|_{L^2(K)} \right)\| h_K^{-1} \sigma_K \|_{L^2(K)}\\
     & \lesssim \left( \sum_{K \in \mathcal{K}}  \| h_K ( \vecb{f} - \nabla p + \nu \Delta_\mathcal{T} \vecb{u}_h) \|_{L^2(K)}^2\right)^{1/2} \| \nabla \vecb{v}_h \|_{L^2}
      + \mathrm{osc}_{q-1}(\vecb{f} - \nabla p,\mathcal{K}) \| \nabla \vecb{v}_h \|_{L^2}\\
     & = \left(\eta_\text{vol}(p,\nabla \vecb{u}) + \mathrm{osc}_{q-1}(\vecb{f} - \nabla p,\mathcal{K})\right) \| \nabla \vecb{v}_h \|_{L^2}.
  \end{align*}
  A division by \(\| \nabla \vecb{v}_h \|_{L^2}\) and the result from (b) conclude the proof of (d).

The proof of (e) is straight forward and employs integration by parts and the
orthogonality of \(\mathrm{div}(\vecb{v}_h)\) onto all \(q_h \in Q_h\) if \(\Pi = 1\) does not satisfy \eqref{eqn:def_reconstoperator}.
Otherwise, if \(\Pi\) satisfies \eqref{eqn:def_reconstoperator}, the assertion follows from \(\mathrm{div}(\Pi \vecb{v}_h) = 0\). 
\end{proof}

\begin{remark}
Theorem~\ref{thm:efficiency}.(b) shows the pressure-dependence also in the efficiency estimate. The volume term \(\eta_\text{vol}(q,\nabla \vecb{u}_h)\) scales with
the term \(\nu^{-1} \| p - q\|_{L^2}\). Hence, a pressure-robust method is only efficient with a good approximation \(q \approx p\).
In the hydrostatic (worst) case with \(\vecb{u}_h = 0\) and \(\vecb{f} = \nabla p\), \(\eta_\text{vol}(q,\nabla \vecb{u})\) is not zero
(hence inefficient with efficiency index infinity) as long as \(q \neq p\) is inserted.
To compute the correct pressure is in general impossible or expensive.
Some strategy to find an approximation that at least yields a higher-order term is discussed
in \cite{MR3366087}.

Note however, that \(\eta_\text{vol}(q,\nabla \vecb{u}_h)\) is efficient for a classical pressure-inrobust method with \(q_h = p_h\) (or some suitable \(H^1\)-approximation),
since then the discrete velocity error and its velocity error
also depends on \(\nu^{-1} \|p - p_h\|_{L^2}\), see e.g.\ our numerical examples in Section~\ref{sec:numerics}.
\end{remark}

\section{Refined residual-based error bounds}\label{novelestimates}
This section offers an alternative a posteriori error estimator and is related to the stream function and vorticity formulation of the Navier--Stokes
equations. The analysis employs the two-dimensional curl operators for vector and scalar fields
\begin{align*}
   \mathrm{curl} \vecb{\phi} := (   \partial \phi_2 / \partial x  -\partial \phi_1 / \partial y) 
  \quad &\text{for } \vecb{\phi}=(\phi_1,\phi_2) \in H^1(\Omega)^2,\\
     \mathrm{curl} \phi := \begin{pmatrix}    -\partial \phi / \partial y \\ \partial \phi / \partial x  \end{pmatrix}   \quad &\text{for } \phi \in H^1(\Omega).   
\end{align*}
The outcome of this alternative approach is a different volume term that only takes
\(\mathrm{curl}(\vecb{f})\) into account and so automatically cancels the gradient part of
the Helmholtz decomposition. Hence, no knowledge or good approximation of \(\mathbb{P} \vecb{f}\) is needed.
The resulting terms are related to the terms in \cite{MR1645033} where error indicators for
discretisations of the streamline and vorticity formulation were derived. However, our error estimator holds for pressure-robust finite element
methods for the velocity and pressure formulation of the Navier--Stokes equations.

Given a Fortin interpolator \(I\) and a reconstruction operator \(\Pi\) with
\eqref{eqn:def_reconstoperator} (possibly \(\Pi = 1\) for divergence-free finite element methods like
the Scott-Vogelius finite element method),
the novel approach exploits that \(\Pi I \vecb{v}\) for some divergence-free function \(\vecb{v} \in \vecb{V}_0\) is again a divergence-free function in \(L^2_\sigma(\Omega)\). Our analysis needs the following assumption on the two operators
additional to \eqref{eqn:def_reconstoperator} and \eqref{eqn:Fortinprops1}-\eqref{eqn:Fortinprops2}.
\begin{assumption} \label{assumption1}
  For every \(\vecb{v} \in \vecb{V}_0\), the Fortin interpolator \(I\) and the reconstruction operator \(\Pi\) satisfy
\begin{align*}
    \Pi I \vecb{v} \in L^2_\sigma(\Omega) \quad \text{and hence} \quad \int_\Omega (1 - \Pi I) \vecb{v} \cdot \nabla q \dx = 0 \quad \text{for all } q \in H^1(\Omega),
\end{align*}
and the estimate
\begin{align*}
   \int_\Omega (1 - \Pi I) \vecb{v} \cdot \vecb{\theta} \dx
   \lesssim \| \nabla \vecb{v} \|_{L^2} \| h_\mathcal{T}^{2} \mathrm{curl} \vecb{\theta} \|_{L^2} \quad \text{for all } \vecb{\theta} \in H(\mathrm{curl},\Omega).
\end{align*}
\end{assumption}
\begin{theorem}[Novel error estimator for pressure-robust methods]\label{thm:eta_averaging_pr}
For \(\vecb{u}_h\) of \eqref{eqn:modified_solution}
and any \(\sigma \in H^1(\mathcal{T})^{2 \times 2}\) (that approximates or equals \(\nabla \vecb{u}_h\)), the error estimator
\begin{align*}
  \eta_\text{new}(\sigma) := 
 \eta_\text{curl}(\sigma) 
 + \eta_\text{jump}(\sigma)
 + \eta_\text{jump,2}(\sigma)
 + \eta_\text{avg}(\sigma)
 + \eta_\text{cons,1}(\sigma)
\end{align*}
with the subterms
  \begin{align*}
     \eta_\text{curl}(\sigma) & := \| h_\mathcal{T}^2 \mathrm{curl}_\mathcal{T} (\vecb{f} + \nu \mathrm{div}_h \sigma) \|_{L^2} \\
     \eta_\text{jump}(\sigma) & := \| h_\mathcal{E}^{1/2} [\nu \sigma \vecb{n}_E] \|_{L^2(\mathcal{E}^\circ)}\\
     \eta_\text{jump,2}(\sigma) & := \| h_\mathcal{E}^{3/2} [(\vecb{f} + \nu \mathrm{div}_h \sigma) \cdot \vecb{\tau_E}] \|_{L^2(\mathcal{E}^\circ)}\\
     \eta_\text{avg}(\sigma) & := \nu \| \nabla \vecb{u}_h - \sigma \|_{L^2}\\
     \eta_\text{cons,1}(\sigma) & := \| \nu \mathrm{div}_h (\sigma) \circ (1 - \Pi) \|_{V_{0,h}^\star}
  \end{align*}
satisfies
\begin{align*}
   \| r \|_{\vecb{V}_0^\star} \lesssim \eta(q)
   \quad \text{and hence} \quad \| \nabla(\vecb{u} - \vecb{u}_h) \|_{L^2}^2 \lesssim \nu^{-2} \eta(q)^2 + \frac{1}{c_0^2} \| \mathrm{div} \vecb{u}_h \|^2_{L^2} .
\end{align*}
Note in particular, that the volume contribution \(\eta_\text{vol}(q,\sigma)\) from Theorem~\ref{thm:errorbounds} has been replaced by the quantity \(\eta_\text{curl}(\sigma)\)
that is pressure-independent (or \(q\)-independent).
\end{theorem}
\begin{proof}
As in the estimation of \(\| r \|_{\vecb{V}_0^\star}\) in the proof of Theorem~\ref{thm:errorbounds}.(b), we subtract the Fortin interpolation \(I\vecb{v}\)
of any testfunction \(\vecb{v}\) by employing \eqref{eqn:modified_solution}, i.e.
\begin{align*}
  r(v) & = \int_\Omega \vecb{f} \cdot (\vecb{v} - \Pi I \vecb{v}) \dx
  - \nu \int_\Omega \nabla \vecb{u}_h : \nabla(\vecb{v} - I \vecb{v}) \dx.
\end{align*}
Given any \(\sigma \in H^1(\mathcal{T})^{2 \times 2}\), an (element-wise) integration by parts shows
\begin{multline*}
  r(v) = \int_\Omega (\vecb{f} + \nu \mathrm{div}_h \sigma) \cdot (\vecb{v} - \Pi I \vecb{v}) \dx
  + \nu \int_\Omega (\sigma - \nabla \vecb{u}_h) : \nabla (\vecb{v} - I \vecb{v}) \dx\\
   + \nu \sum_{E \in \mathcal{E}^\circ} \int_E [\sigma\vecb{n}] \cdot (\vecb{v} - I \vecb{v}) \ds 
   + \nu \int_\Omega (\mathrm{div}_h \sigma) \cdot (\Pi I \vecb{v} - I\vecb{v}) \dx
   =: A + B + C + D.
\end{multline*}
The terms \(B,C\) and \(D\) are estimated as in Theorem~\ref{thm:errorbounds}.(b) by
\begin{align*}
  B & := \nu \sum_{T \in \mathcal{T}}
  \int_T (\sigma - \nabla \vecb{u}_h): \nabla (\vecb{v} - I \vecb{v}) \dx
  \leq \nu \| \sigma - \nabla \vecb{u}_h \|_{L^2} \| \nabla \vecb{v} \|_{L^2}\\
  C & := \nu \sum_{E \in \mathcal{E}^\circ}
  \int_E [\sigma\vecb{n}] \cdot (\vecb{v} - I \vecb{v}) \ds
  \leq \nu \|h_\mathcal{E}^{1/2}  [\sigma\vecb{n}] \|_{L^2( \mathcal{E}^\circ)} \| \nabla \vecb{v} \|_{L^2}\\
  D &:= \nu \int_\Omega (\mathrm{div}_h \sigma) \cdot (\Pi I \vecb{v} - I\vecb{v}) \dx
  \leq \nu \| (\mathrm{div}_h \sigma) \circ (1 - \Pi) \|_{V_{0,h}^\star} \| \nabla \vecb{v} \|_{L^2}.
\end{align*}
It remains to estimate term \(A\). As $\vecb{v} - \Pi I \vecb{v}$ is exactly divergence free and has a zero normal trace we can apply Theorem 3.1, chapter 1 in \cite{GR86} to find a scalar potential $\psi \in H^1_0(\Omega)$ with $\mathrm{curl} \psi = \vecb{v} - \Pi I \vecb{v}$.
In the following we bound the weighted $L^2$ norm of $\psi$.  Note that from $ h_{\mathcal{T}}^{-4} \psi \in L^2(\Omega)$ follows $h_{\mathcal{T}}^{-2} \psi \in h_{\mathcal{T}}^{2} \mathrm{curl}( H(\mathrm{curl}, \Omega))$, due to the surjectivity of the $\mathrm{curl}$ operator (de Rham complex) and so
\begin{align*}
  \| h_{\mathcal{T}}^{-2} \psi \|_{L^2(\Omega)} &= \frac{\int_{\Omega} h_{\mathcal{T}}^{-2}  \psi  h_{\mathcal{T}}^{-2} \psi  \dx}{\|  h_{\mathcal{T}}^{-2} \psi  \|_{L^2(\Omega)}} \\
  &\le \sup\limits_{\vecb{\theta} \in H(\mathrm{curl}, \Omega)} \frac{\int_{\Omega}  h_{\mathcal{T}}^{-2} \psi  h_{\mathcal{T}}^{2} \mathrm{curl} \vecb{\theta}\dx}{\| h_{\mathcal{T}}^{2} \mathrm{curl}\vecb{\theta} \|_{L^2(\Omega)} } = \sup\limits_{\vecb{\theta} \in H(\mathrm{curl}, \Omega)} \frac{\int_{\Omega}   \psi \mathrm{curl} \vecb{\theta}\dx}{\| h_{\mathcal{T}}^{2} \mathrm{curl}\vecb{\theta} \|_{L^2(\Omega)} }.
\end{align*}
On the other hand one can bound the supremum by $\| h_{\mathcal{T}}^{-2}  \psi \|_{L^2(\Omega)}$ with a simple Cauchy Schwarz estimate.  Using Assumption~\ref{assumption1} it follows by an integration by parts and $\psi \in H^1_0(\Omega)$ that 
\begin{align} \label{eqn:psiltwonorm}
  \| h_{\mathcal{T}}^{-2}  \psi \|_{L^2(\Omega)} & = \sup\limits_{\vecb{\theta} \in H(\mathrm{curl}, \Omega)} \frac{\int_{\Omega}  \psi \mathrm{curl} \vecb{\theta} \dx}{\| h_{\mathcal{T}}^{2}  \mathrm{curl} \vecb{\theta} \|_{L^2(\Omega)} } \\
  &= \sup\limits_{\vecb{\theta} \in H(\mathrm{curl}, \Omega)} \frac{\int_{\Omega}  \mathrm{curl} \psi \cdot  \vecb{\theta} \dx}{\|h_{\mathcal{T}}^{2}  \mathrm{curl} \vecb{\theta} \|_{L^2(\Omega)} } \lesssim  \| \nabla \vecb{v} \|_{L^2(\Omega)}. \nonumber 
\end{align}
  With $\vecb{\theta}_h := \vecb{f} + \nu \mathrm{div}_h \sigma$ and $\psi = 0 $ on $\partial \Omega$ a piecewise integration by parts yields
\begin{align*}
  A &:= \int_\Omega \vecb{\theta}_h \cdot (\vecb{v} - \Pi I \vecb{v}) \dx = \int_\Omega \vecb{\theta}_h \cdot \mathrm{curl} \psi \dx\\
                                                                    &= \sum_{T \in \mathcal{T}} \int_T \mathrm{curl} \vecb{\theta}_h  \psi \dx + \sum_{E \in \mathcal{E}^\circ} \int_E [\vecb{\theta}_h \cdot \vecb{\tau_E}] \psi  \ds\\
                                                                    & \lesssim \sum_{T \in \mathcal{T}} \| h_T^2 \mathrm{curl} \vecb{\theta}_h \|_{L^2(T)} \| h_T^{-2} \psi \|_{L^2(T)} + \sum_{E \in \mathcal{E}^\circ} \| h_E^{3/2} [\vecb{\theta}_h \cdot \vecb{\tau_E}]\|_{L^2(E)} \| h_E^{- 3/2} \psi  \|_{L^2(E)}\\
  &\lesssim \left( \| h_\mathcal{T}^2 \mathrm{curl}_\mathcal{T} \vecb{\theta}_h \|_{L^2(\Omega)} + \| h^{3/2} [\vecb{\theta}_h \cdot \vecb{\tau_E}]\|_{L^2(\mathcal{E}^\circ)} \right) \left( \| h_\mathcal{T}^{-2} \psi \|_{L^2(\Omega)} + \|  h_\mathcal{E}^{- 3/2} \psi  \|_{L^2(\mathcal{E}^\circ)} \right).
\end{align*}
Using a standard scaling argument we get, for each edge $E \in \mathcal{E}^\circ$,
\begin{align*}
\|  h_E^{- 3/2} \psi  \|_{L^2(E)} \lesssim h_T^{- 2}  \|\psi\|_{L^2(T)} +  h_T^{-1}  \| \nabla \psi  \|_{L^2(T)}.
\end{align*}
For the second term in the previous estimate we have
\begin{align*}
h_T^{-1}  \| \nabla \psi  \|_{L^2(T)} = h_T^{-1}  \| \mathrm{curl} \psi  \|_{L^2(T)} = h_T^{-1}  \|  \vecb{v} - \Pi I \vecb{v} \|_{L^2(T)} \lesssim  \| \nabla \vecb{v} \|_{L^2(\omega_T)}.
\end{align*}
Together with \eqref{eqn:psiltwonorm} and an overlap argument this leads to 
\begin{align*}
  \|  h_\mathcal{T}^{- 3/2} \psi  \|_{L^2(\mathcal{E}^\circ)} &\lesssim  \|h_\mathcal{T}^{- 2}  \psi  \|_{L^2(\Omega)} + \|h_\mathcal{T}^{-1}  \nabla \psi  \|_{L^2(\Omega)} \lesssim \| \nabla \vecb{v} \|_{L^2(\Omega)}.
\end{align*}
This concludes the estimate for $A$, i.e.
\begin{align*}
  A \lesssim \left( \eta_\text{curl}(\sigma) + \eta_\text{jump,2}(\sigma) \right) \| \nabla \vecb{v} \|_{L^2(\Omega)}.
\end{align*}
The collection of all separate estimates for \(A\) to \(D\) shows
\begin{align*} 
 r(v) \lesssim \eta(\sigma) \| \nabla \vecb{v} \|_{L^2}
\end{align*}
and a division by \(\| \nabla \vecb{v} \|_{L^2}\) concludes the proof. 
\end{proof}

The same techniques also a yield a novel error estimate for classical methods.
\begin{proposition}[Novel error estimator for classical methods]\label{thm:eta_averaging_cl}
For \(\vecb{u}_h\) of \eqref{eqn:classical_solution}
and any \(\sigma \in H^1(\mathcal{T})^{2 \times 2}\) (that approximates or equals \(\nabla \vecb{u}_h\)),
the error estimator
\begin{multline*}
  \eta_\text{new}(\sigma) := 
  \eta_\text{curl}(\sigma) 
 + \eta_\text{jump}(\sigma)
 + \eta_\text{avg}(\sigma)
 + \| (\vecb{f} + \nu \mathrm{div}_h \sigma) \circ (1 - \Pi) \|_{V_{0,h}^\star}
\end{multline*}
satisfies
\begin{align*}
   \| r \|_{\vecb{V}_0^\star} \lesssim \eta(q)
   \quad \text{and hence} \quad \| \nabla(\vecb{u} - \vecb{u}_h) \|_{L^2}^2 \lesssim \nu^{-2} \eta(q)^2 + \frac{1}{c_0^2} \| \mathrm{div} \vecb{u}_h \|^2_{L^2}.
\end{align*}
Note, that \(\Pi\) is used only in the error estimator here, but not in the calculation of \(\vecb{u}_h\). It is not allowed to set \(\Pi = 1\)
if the classical method is not divergence-free, i.e. \(\Pi\) has to satisfy \eqref{eqn:def_reconstoperator}. The difference to the previous theorem lies in the
appearence of \(\vecb{f}\) in the consistency error \(\| (\vecb{f} + \nu \mathrm{div}_h \sigma) \circ (1 - \Pi) \|_{V_{0,h}^\star}\).
\end{proposition}
\begin{proof}
The proof follows the proof of Theorem~\ref{thm:eta_averaging_pr} but one has to add the
term \(\int_\Omega \vecb{f} \cdot (\vecb{v} - \Pi \vecb{v}) \dx\) which can be added to the estimate of term \(C\). 
\end{proof}

The next theorem establishes the efficiency of the novel terms \(\eta_\text{curl}(\sigma)\)
and \(\eta_\text{jump,2}(\sigma)\) for \(\sigma = \nabla \vecb{u}_h\). For the efficiency
of the other terms see Theorem~\ref{thm:efficiency}.

\begin{theorem}[Efficiency for \(\sigma = \nabla \vecb{u}_h\)]\label{thm:efficiency_new_estimate}
  It holds
   \begin{itemize}
   \item[(a)] \( \nu^{-1} h_T^2 \| \mathrm{curl}_\mathcal{T}(\vecb{f} + \nu \Delta_\mathcal{T} \vecb{u}_h) \|_{L^2(T)} \lesssim  \|\nabla ( \vecb{u} - \vecb{u}_h) \|_{L^2(T)} \)\\
    \hspace*{5cm} \( + \nu^{-1} h_T \mathrm{osc}_k(\mathrm{curl}(\vecb{f} + \nu \Delta_\mathcal{T} \vecb{u}_h),T),\)
  \item[(b)] \( \nu^{-1} h_E^{3/2} \| [(\vecb{f} + \nu \Delta_\mathcal{T} \vecb{u}_h) \cdot \vecb{\tau_E}] \|_{L^2(E)} \lesssim  \|\nabla ( \vecb{u} - \vecb{u}_h) \|_{L^2(\omega_E)}\) \\
  \hspace*{5cm} \(+ \nu^{-1} h_E \mathrm{osc}_k(\mathrm{curl}_\mathcal{T}(\vecb{f} + \nu \Delta_\mathcal{T} \vecb{u}_h),\mathcal{T}(E))\)\\
  \hspace*{5cm} \(+ \nu^{-1} h_E \mathrm{osc}_k([f \cdot \vecb{\tau_E}],E) + \mathrm{osc_k}(\vecb{f} - \nabla p,\mathcal{T}(E)),\)
  \end{itemize}
  for all \(T \in \mathcal{T}\) and \(E \in \mathcal{E}^\circ\).
\end{theorem}
\begin{proof}
The proof employs the standard Verf\"urth bubble-technique. To shorten the notion in the proof
of (a), we define
   \begin{align*}
      Q_T := \mathrm{curl}(\vecb{f} + \nu \Delta_\mathcal{T} \vecb{u}_h)|_T \quad \text{for any } T \in \mathcal{T}.
   \end{align*}
Then, it holds (similarly to \cite{MR1213837})
\begin{align*}
   \| &\pi_{P_k(T)} Q_{T} \|_{L^2(T)} \\
   & \lesssim \sup_{\vecb{v}_T \in P_k(T)^2} \int_T \pi_{P_k(T)} Q_{T} \cdot (b_T^2 \vecb{v}_T) \dx / \| \vecb{v}_T \|_{L^2(T)}\\
   & \leq \sup_{\vecb{v}_T \in P_k(T)^2} \frac{\int_T Q_{T} b_T^2 \vecb{v}_T \dx}{\| \vecb{v}_T \|_{L^2(T)}}
   + \sup_{\vecb{v}_T \in P_k(T)^2} \frac{\|Q_{T} - \pi_{P_k(T)} Q_{T}\|_{L^2(T)} \| b_T^2 \vecb{v}_T \|_{L^2(T)}}{ \| \vecb{v}_T \|_{L^2(T)}} .
\end{align*}
Testing the continuous system with the (divergence-free) testfunction \(\mathrm{curl} (b_T^2 \vecb{v}_T) \in H^2(T)^2 \cap H^1_0(\Omega)^2\) and an integration by parts leads to
\begin{align*}
  \int_T Q_{T} b_T^2 \vecb{v}_T \dx 
  & = \int_T (\vecb{f} + \nu \Delta_\mathcal{T} \vecb{u}_h) \cdot \mathrm{curl} (b_T^2 \vecb{v}_T) \dx\\
  & = \int_T \nu \nabla ( \vecb{u} - \vecb{u}_h) : \nabla \mathrm{curl} (b_T^2 \vecb{v}_T) \dx\\
  & \leq \nu \| \nabla ( \vecb{u} - \vecb{u}_h) \|_{L^2(T)} \| \nabla \mathrm{curl} (b_T^2 \vecb{v}_T) \|_{L^2(T)}.
\end{align*}
A discrete inverse inequality shows \(\| \nabla \mathrm{curl} (b_T^2 \vecb{v}_T) \|_{L^2(T)} \lesssim h_T^{-2} \| b_T^2 \vecb{v}_T \|_{L^2(T)}\). 
This and the norm equivalence \(\| b_T^2 \vecb{v}_T \|_{L^2(T)} \approx \| \vecb{v}_T \|_{L^2(T)}\) lead to
\begin{align*}
  h_T^2 \| \pi_{P_k(T)} Q_{T} \|_{L^2(T)} \lesssim  \nu \|\nabla ( \vecb{u} - \vecb{u}_h) \|_{L^2(T)} + h_T^2 \|Q_{T} - \pi_{P_k(T)} Q_{T}\|_{L^2(T)}.
\end{align*}
This concludes the proof of (a).

In the proof of (b), we use the notation
\begin{align*}
  Q_E := [\vecb{f} + \nu \Delta_\mathcal{T} \vecb{u}_h] \cdot \vecb{\tau_E} \quad \text{for any } E \in \mathcal{E}
\end{align*}
and the face bubble \(b_E\) with support \(\omega_E\) for every face \(E \in \mathcal{E}\).
Then,
\begin{align*}
  \| \pi_{P_k(E)} Q_E \|_{L^2(E)} \lesssim \sup_{\vecb{v}_E \in P_k(E)^2} \frac{\int_E Q_E \cdot (b_E^2 \vecb{v}_E) \ds}{ \| \vecb{v}_E \|_{L^2(E)}}
  + \| Q_E - \pi_{P_k(E)} Q_E \|_{L^2(E)}.
\end{align*}
Testing the continuous equation with the divergence-free testfunction \(\mathrm{curl} (b_E^2 \vecb{v}_E) \in H^1_0(\Omega)\) (where \(\vecb{v}_E\) is reasonably extended to \(\omega_E\))
and an integration by parts show
\begin{align*}
  \int_E Q_E & \cdot (b_E^2 \vecb{v}_E) \ds\\
  & = \int_E [(\vecb{f} + \nu \Delta_\mathcal{T} \vecb{u}_h) \cdot \vecb{\tau_E}] \cdot (b_E^2 \vecb{v}_E) \ds\\
  & = \int_{\omega_E} (\vecb{f} + \nu \Delta_\mathcal{T} \vecb{u}_h) : \mathrm{curl}(b_E^2 \vecb{v}_E) \dx
  - \int_{\omega_E} \mathrm{curl}(\vecb{f} + \nu \Delta_\mathcal{T} \vecb{u}_h) : (b_E^2 \vecb{v}_E) \dx\\
  & \leq \| \vecb{f} + \nu \Delta_\mathcal{T} \vecb{u}_h \|_{L^2(\omega_E)} \| \mathrm{curl} (b_E^2 \vecb{v}_E) \|_{L^2(\omega_E)}
  + \| Q_T \|_{L^2(\omega_E)} \| b_E^2 \vecb{v}_E \|_{L^2(\omega_E)}.
\end{align*}
A discrete inverse inequality \(\| \mathrm{curl} (b_E^2 \vecb{v}_E) \|_{L^2(\omega_E)} \lesssim h_T^{-1} \| b_E^2 \vecb{v}_E \|_{L^2(\omega_E)}\)
and a scaling argument (see \cite{MR1213837}), that yields \(\| b_E^2 \vecb{v}_E \|_{L^2(\omega_E)} \lesssim h_T^{1/2} \| \vecb{v}_E \|_{L^2(E)}\), show
\begin{align*}
  h_E^{3/2} \| \pi_{P_k(E)} Q_{E} \|_{L^2(E)} \lesssim & h_T \| \vecb{f} + \nu \Delta_\mathcal{T} \vecb{u}_h \|_{L^2(\omega_E)} + h_T^{2}\| Q_T \|_{L^2(\omega_E)} \\
  &+ h_E^{3/2} \| Q_E - \pi_{P_k(E)} Q_E \|_{L^2(E)}.
\end{align*}
The proof of Theorem~\ref{thm:efficiency}.(c) yields
\begin{align*}
  \| \vecb{f} + \nu \Delta_\mathcal{T} \vecb{u}_h \|_{L^2(\omega_E)}
  \lesssim \nu \| \nabla (\vecb{u} - \vecb{u}_h) \|_{L^2} + \mathrm{osc_k}(\vecb{f} - \nabla q,\mathcal{T}(E)).
\end{align*}
This and the already proven result from (a) conclude the proof. 
\end{proof}

%

\section{Proof of Assumption~\ref{assumption1} for certain finite element methods}\label{sec:assumptionproof}
This section proves Assumption~\ref{assumption1} for certain finite element methods. For the analysis several standard interpolation operators that are related to the de
Rahm complex (see e.g.\ \cite{MR2373173}) are employed. These are a (projection based) nodal interpolation operator \( \IL  \), the lowest order Raviart-Thomas interpolation
operator \( \IR \) and the lowest-order N\'ed\'elec interpolation operator \(\IN\). These operators satisfy in particular the commuting diagram properties in two dimensions (see \cite{MR1746160}) 
\begin{align}\label{eqn:commuting_props}
   \mathrm{curl} (\IL  \vecb{v}) = \IR (\mathrm{curl} \vecb{v})
   \quad \text{and} \quad
   \nabla (\IL  \vecb{v}) = \IN  (\nabla \vecb{v})
\end{align}
for arbitrary sufficiently smooth functions \(\vecb{v}\). Furthermore we need a refined Helmholtz decomposition.
\begin{lemma}[\cite{MR2373173}]\label{lem:decomp_curlestimate}
It exists an operator \( \ClemN : H(\mathrm{curl},\Omega) \rightarrow \mathcal{N}_0(\mathcal{T})\) with the property:
for every \(\vecb{\theta} \in H(\mathrm{curl},\Omega)\)
exists a decomposition
\begin{align*}
   \vecb{\theta} - \ClemN \vecb{\theta} = \nabla \phi + \vecb{y}
\end{align*}
with \(\phi \in H^1(\Omega)\), \(\vecb{y} \in H^1(\Omega)^2\), and
\begin{align*}
  h_T^{-1} \| \vecb{y} \|_{L^2(T)} + \| \nabla \vecb{y} \|_{L^2(T)} \lesssim \| \mathrm{curl} \vecb{\theta} \|_{L^2(T)} \quad \text{for all } T \in \mathcal{T}.
\end{align*}
\end{lemma}
\begin{proof}
In \cite{MR2373173} a proof for three dimensions is given. The two dimensional case follows similarly. 
\end{proof}

\begin{lemma}[Regular decomposition]\label{lem:reg_decomposition}
  For each $\vecb{\theta} \in H(\mathrm{curl}, \omega)$ there exists a decomposition with $\alpha \in H^2(\omega)$ and $\vecb{\beta} \in H^1(\omega)^2$ such that
  \begin{align*}
      \vecb{\theta} = \nabla \alpha + \vecb{\beta},
  \end{align*}
  with
  \begin{align*}
    || \nabla \vecb{\beta} ||_{L^2(\omega)} \lesssim || \mathrm{curl} \vecb{\theta}||_{L^2(\omega)} \quad \textrm{and} \quad \int_\omega \vecb{\beta} \dx = 0.
  \end{align*}
\end{lemma}
\begin{proof}
  Let $q := \mathrm{curl}\vecb{\theta}$ and $\tilde{\omega}$ be a convex domain such that $\omega \subset \tilde{\omega}$. We define $\tilde{q}$ as a trivial extension of $q$ by zero, i.e.\ $\tilde{q}|_\omega = q$ and $\tilde{q}|_{\tilde{\omega} \setminus \omega } = 0$. In the next step we seek the solution $w \in H^1(\tilde{\omega})$ of the Poisson problem $ \Delta w = \mathrm{curl} \mathrm{curl} w = \tilde{q}$ on $\tilde{\omega}$. Using a regularity estimate for the Poisson problem on the convex domain $\tilde{\omega}$, it follows for $\tilde{\vecb{\beta}}:= \mathrm{curl} w$ and $\vecb{\beta} := \tilde{\vecb{\beta}}|_\omega - \int_\omega \tilde{\vecb{\beta}} \dx / |\omega| $
  that
  \begin{align*}
    || \nabla \vecb{\beta} ||_{L^2(\omega)} \lesssim || \nabla \tilde{\vecb{\beta}} ||_{L^2(\tilde{\omega})} \lesssim || w ||_{H^2(\tilde{\omega})} \lesssim || \tilde{q} ||_{L^2(\tilde{\omega})} = ||q ||_{L^2(\omega)} =  || \mathrm{curl} \vecb{\theta} ||_{L^2(\omega)}.
  \end{align*}
  Since $\mathrm{curl}(\vecb{\theta} - \vecb{\beta}) =0$ in \(\omega\), its exists a vector potential $\alpha \in H^2(\omega)$ such that $\nabla \alpha =\vecb{\theta} - \vecb{\beta}$. This concludes the proof. 
\end{proof}

\begin{theorem}[Proof of Assumption~\ref{assumption1} for finite element methods with \(P_0\) pressure space]
  If the reconstruction operator \(\Pi\) and the Fortin operator \(I\) satisfy \eqref{eqn:reconstop_local_orthogonalities} and
  \begin{align}\label{eq:P2P0_I2_property}
     \int_E (1 - I) \vecb{v} \cdot \vec{n}_E \ds = \int_E (1 - \Pi I) \vecb{v} \cdot \vec{n}_E \ds = 0 \quad \text{for all } E \in \mathcal{E},
  \end{align}
  also Assumption~\ref{assumption1} is satisfied.
\end{theorem}
\begin{remark}
Condition \eqref{eq:P2P0_I2_property} is satisfied for the Forint interpolators  for the
\(P_2 \times P_0\), \(P_3 \times P_0\) and the Bernardi--Raugel finite element methods \cite[Section 8.4.3]{MR3097958}.
For these methods the reconstruction operator \(\Pi\) is the standard interpolation into the space \(BDM_{1}\) or \(RT_0\) \cite{LM2016}.
\end{remark}
\begin{proof}
  Since every function \(\vecb{g} \in H^1(T)\) with \(\int_E \vecb{g} \cdot \vecb{n} \ds = 0\) along all edges \(E \in \mathcal{E}(T)\) of \(T\) satisfies a discrete Friedrichs inequality
  \(\| \vecb{g} \|_{L^2(T)} \lesssim h_T \| \nabla \vecb{g} \|_{L^2(T)}\), see e.g.\ \cite{533fbff7}, it follows together with \eqref{eqn:reconstop_local_orthogonalities}
  \begin{align*}
     \| (1 - \Pi I) \vecb{v} \|_{L^2(T)} 
      & \leq \| (1 - I) \vecb{v}\|_{L^2(T)} + \| (1-\Pi)(I \vecb{v})\|_{L^2(T)}\\
      & \lesssim h_T \| \nabla \vecb{v}\|_{L^2(T)} + h_T \| \nabla I \vecb{v} \|_{L^2(T)}
      \lesssim h_T \| \nabla \vecb{v}\|_{L^2(T)}.
  \end{align*}
  Since \((1 - \Pi I) \vecb{v}\) is divergence-free, it holds \((1 - \Pi I) \vecb{v} = \mathrm{curl} \psi\) for some \(\psi \in H^1_0(\Omega) \cap H^2(\Omega)\), see Corollary 3.2 in \cite{GR86}.
  Condition \eqref{eq:P2P0_I2_property} implies that the standard interpolator into \(RT_0\) vanishes, i.e.\ \(\IR \mathrm{curl} \psi = 0\). Moreover, by the commuting properties 
  \eqref{eqn:commuting_props} of the de Rham complex, it also holds \(\mathrm{curl} (\IL \psi) = \IR \mathrm{curl} \psi = 0\).
  An integration by parts and standard interpolation estimates yield
  \begin{align*}
    \int_\Omega \vecb{\theta} \cdot (1 - \Pi I) \vecb{v} \dx &= \int_\Omega \vecb{\theta} \cdot  \mathrm{curl} (\psi - \IL \psi) \dx = \int_\Omega\mathrm{curl} \vecb{\theta} \cdot (\psi - \IL \psi) \dx \\
    & \leq \| h_\mathcal{T}^{2} \mathrm{curl} \vecb{\theta} \|_{L^2}  \| h_\mathcal{T}^{-2} (\psi - \IL \psi) \|_{L^2} \\
    & \leq \| h_\mathcal{T}^{2} \mathrm{curl} \vecb{\theta} \|_{L^2} \| h_\mathcal{T}^{-1} \nabla(\psi - \IL \psi) \|_{L^2}\\
     & = \| h_\mathcal{T}^{2} \mathrm{curl} \vecb{\theta} \|_{L^2} \| h_\mathcal{T}^{-1} \mathrm{curl}(\psi) \|_{L^2} \\
    & = \| h_\mathcal{T}^{2} \mathrm{curl} \vecb{\theta} \|_{L^2} \| h_\mathcal{T}^{-1} (1 - \Pi I) \vecb{v} \|_{L^2} \leq \| h_\mathcal{T}^{2} \mathrm{curl} \vecb{\theta} \|_{L^2} \| \nabla \vecb{v} \|_{L^2},
  \end{align*}
  where we used that the curl is just the rotated gradient in two dimensions.
  This concludes the proof. 
\end{proof}

\begin{theorem}[Proof of Assumption~\ref{assumption1} for finite element methods with discontinuous \(P_1\) pressure space]\label{thm:proof_asssumption_P1dc}
  If the reconstruction operator \(\Pi\) and the Fortin operator \(I\) satisfy
  \begin{align}\label{eq:P2+P1_I2_property}
     \int_T (1 - I) \vecb{v} \dx = \int_T (1 - \Pi I) \vecb{v} \dx = 0 \quad \text{for all } T \in \mathcal{T},
  \end{align}
  also Assumption~\ref{assumption1} is satisfied.
\end{theorem}
\begin{remark}
  Condition \eqref{eq:P2+P1_I2_property} is satisfied by the \(P_2\)-bubble finite element method
  and its Fortin interpolator \cite[Section 8.6.2]{MR3097958}. 
  A suitable reconstruction operator \(\Pi\) is the standard interpolation into the space \(BDM_{2}\) or \(RT_1\) \cite{LMT15,LM2016}.
  Moreover, the result generalises to all \(P_k \times P_{k-2}\) finite element methods with \(k > 2\).
\end{remark}

\begin{proof}
  A triangle inequality, interpolation properties of \(\Pi\), a Poincar\'e inequality, and the \(H^1\)-stability of \(I\) show
  \begin{align*}
     \| (1 - \Pi I) \vecb{v} \|_{L^2(T)} \leq  \| (1 - \Pi) I\vecb{v} \|_{L^2(T)} + \| (1 - I) \vecb{v} \|_{L^2(T)} \lesssim h_T \| \nabla \vecb{v}\|_{L^2(T)}.
  \end{align*}
  To estimate the dual norm, Lemma~\ref{lem:decomp_curlestimate} yields \(\vecb{\theta} - \ClemN \vecb{\theta} = \nabla \phi + \vecb{y}\)
  with 
  \begin{align*}
    \| h_\mathcal{T} \vecb{y} \|_{L^2} \lesssim \|h^2_\mathcal{T} \mathrm{curl}\vecb{\theta} \|_{L^2}. 
  \end{align*}
Also note that due to $\ClemN \vecb{\theta} \in  H(\mathrm{curl},T)$  we can use the regular decomposition from Lemma \ref{lem:reg_decomposition} to find
\begin{align*}
    \ClemN \vecb{\theta}|_T = \nabla \alpha_T + \vecb{\beta}_T \quad \text{for all } T \in \mathcal{T}
\end{align*}
with some $\alpha_T \in H^2(T)$ and $\vecb{\beta}_T \in [H^1(T)]^2$ such that \(\int_{T} \vecb{\beta}_T \dx = 0\) and
$$\| \nabla \vecb{\beta}_T \|_{L^2(T)} \lesssim \| \mathrm{curl} (\ClemN \vecb{\theta}) \|_{L^2(T)} \lesssim \| \mathrm{curl} \vecb{\theta} \|_{L^2(T)}.$$
Together with the projection property of \(\IN\), the commuting properties \eqref{eqn:commuting_props} of the de Rham complex and
the continuity of the nodal interpolation $\IL$ for $H^2$ functions, the Helmholtz decomposition can be cast into the discrete version
\begin{align*}
    \ClemN \vecb{\theta}|_T =  \IN (\nabla \alpha_T + \vecb{\beta}_T) = \nabla (\IL \alpha_T) + \IN \vecb{\beta}_T.
\end{align*}
The combination of all decompositions defines some function \(\alpha_\mathcal{T} \in P_1(\mathcal{T})\) and \(\beta_\mathcal{T} \in P_1(\mathcal{T})^2\)
with
$$
  \| h_\mathcal{T}^2 \nabla_h \vecb{\beta}_\mathcal{T} \|_{L^2} \lesssim \| h_\mathcal{T}^2 \mathrm{curl} \vecb{\theta} \|_{L^2}.
$$
  Since \(\vecb{z} := (1 - \Pi I) \vecb{v}\) is orthogonal onto piecewise constants (by \eqref{eq:P2+P1_I2_property}), in particular
  the piecewise constant function \(\nabla (\IL \alpha)_\mathcal{T} \in P_0(\mathcal{T})^2\),  
  and gradients (because \(\vecb{z}\) is divergence-free and has zero boundary data), it follows
  \begin{align*}
    \int_\Omega \vecb{\theta} \cdot (1 - \Pi I) \vecb{v} \dx = \int_\Omega \vecb{z} \cdot \vecb{\theta} \dx   &= \int_\Omega \vecb{z} \cdot (\vecb{\theta} - \Pi_{ND_0} \vecb{\theta}) \dx
                                                               + \int_\Omega \vecb{z} \cdot \Pi_{ND_0} \vecb{\theta} \dx \\
                                                             &=\int_\Omega \vecb{z} \cdot \vecb{y} \dx + \int_\Omega \vecb{z} \cdot \vecb{\beta}_\mathcal{T} \dx \\
                                                             &= \int_\Omega h_\mathcal{T}^{-1} \vecb{z} \cdot h_\mathcal{T} \vecb{y} \dx + \int_\Omega h_\mathcal{T}^{-1} \vecb{z} \cdot h_\mathcal{T} \vecb{\beta}_\mathcal{T} \dx \\
                                                             & \lesssim  \|h_\mathcal{T}^{-1} \vecb{z} \|_{L^2} (\| h_\mathcal{T} \vecb{y} \|_{L^2} + \| h_\mathcal{T}^2 \nabla_h \vecb{\beta}_\mathcal{T} \|_{L^2} ) \\
    &\lesssim \|h_\mathcal{T}^{-1} \vecb{z} \|_{L^2}\| h_\mathcal{T}^{2} \mathrm{curl} \vecb{\theta} \|_{L^2} \lesssim \| \nabla \vecb{v} \|_{L^2}\| h_\mathcal{T}^{2} \mathrm{curl} \vecb{\theta} \|_{L^2}.
  \end{align*}
  Note, that we used an elementwise Poincar\'e inequality for \(\vecb{\beta}_\mathcal{T}\) (which has piecewise integral mean zero). This concludes the proof. 
\end{proof}

\begin{theorem}[Proof of Assumption~\ref{assumption1} for the mini finite element method]\label{thm:proof_asssumption_mini}
  The mini finite element method family with the reconstruction operator from \cite{2016arXiv160903701L} and a Fortin operator \(I\) with
  the property (see e.g.\ \cite[Section 8.4.2]{MR3097958})
  \begin{align}\label{eq:mini_I2_property}
     \int_T (1 - I) \vecb{v} \ds = 0 \quad \text{for all } ~ T \in \mathcal{T}
  \end{align}
  satisfies Assumption~\ref{assumption1}.
\end{theorem}

\begin{proof}
  For the mini finite element method, the reconstruction operator is given in \cite{2016arXiv160903701L}.
  It in particular satisfies \eqref{eqn:reconstop_local_orthogonalities} in the sense
  \begin{align}\label{def:MINIdecomp}
     (1-\Pi) I \vecb{v} = \sum_{y \in \mathcal{N}} \vecb{\sigma}_y
  \end{align}
  where \(\vecb{\sigma}_y \in BDM_{2}(\mathcal{T}(\omega_y))\) satisfies \(\| \vecb{\sigma}_y \|_{L^2(\omega_y)} \lesssim h_y \|\nabla I \vecb{v} \|_{L^2(\omega_y)}\)
  on the nodal patch \(\omega_y\) of the node \(y \in \mathcal{N}\) and
  (at least) the local orthogonality
  \begin{align*}
     \int_{\omega_y} \vecb{\sigma}_y \dx = 0.
  \end{align*}
  Furthermore we have $\vecb{\sigma}_y \cdot n = 0$ on the boundary $\partial \omega_y$. This time, the operators \(I\) and \(\Pi\) do not share the same orthogonality on cell-wise constants as in Theorem~\ref{thm:proof_asssumption_P1dc}, but one can split up the
  \(L^2\)-norm by a triangle inequality
  \begin{align*}
    \| (1 - \Pi I) \vecb{v} \|_{L^2} \leq \| (1- \Pi) I\vecb{v} \|_{L^2} + \| (1- I) \vecb{v} \|_{L^2}.
  \end{align*}
  Due to \eqref{eq:mini_I2_property} the norm \(\| (1- I) \vecb{v} \|_{L^2(T)}\)  can be estimated
  as in Theorem~\ref{thm:proof_asssumption_P1dc} and it remains to estimate \(\| (1- \Pi) I\vecb{v} \|_{L^2}\). 
  For the first one, it holds
  \begin{align*}
    \| (1- \Pi) I \vecb{v} \|_{L^2(T)}^2 
    = \sum_{z \in \mathcal{N}(T)} \int_T \vecb{\sigma} _z \cdot (1-\Pi) I \vecb{v} \dx\\
    \leq \sum_{z \in \mathcal{N}(T)} \| \vecb{\sigma}_z \|_{L^2(\omega_z)} \| (1-\Pi) I \vecb{v} \|_{L^2(T)}\\
    \leq h_T \|\nabla I \vecb{v} \|_{L^2(\omega_T)} \| (1-\Pi) I \vecb{v}_h \|_{L^2(T)}
  \end{align*}
  and hence
  \begin{align*}
  \| (1- \Pi) I \vecb{v} \|_{L^2(T)} \lesssim h_T \|\nabla I \vecb{v} \|_{L^2(\omega_T)} \lesssim h_T \| \nabla \vecb{v} \|_{L^2(\omega_T)}.
\end{align*}
For the estimate of the dual norm, inserting the decomposition from Lemma~\ref{lem:decomp_curlestimate} leads to
\begin{align*}
\int_\Omega \vecb{\theta} \cdot (1 - \Pi I) \vecb{v} \dx = \int_\Omega \vecb{z} \cdot \vecb{\theta} \dx   = \int_\Omega \vecb{z} \cdot \vecb{y} \dx + \int_\Omega \vecb{z} \cdot \ClemN \vecb{\theta} \dx.
\end{align*}
The first integral can be estimated as in Theorem~\ref{thm:proof_asssumption_P1dc} and it remains to estimate the second integral where we employ
the decomposition \eqref{def:MINIdecomp} for \((1 - \Pi) I \vecb{v} = \sum_{y \in \mathcal{N}} \vecb{\sigma}_y \) and its orthogonality properties, i.e.
\begin{align}\label{eqn:intemediate_mini}
  \int_\Omega \vecb{z} \cdot \ClemN \vecb{\theta}  \dx
  = \int_\Omega (1 - I) \vecb{v} \cdot \ClemN \vecb{\theta} \dx +  \sum_{y \in \mathcal{N}} \int_{\omega_y} h_y^{-1} \vecb{\sigma}_y \cdot h_y \ClemN \vecb{\theta} \dx
\end{align}
and we bound both integrals separately.
The first integral of \eqref{eqn:intemediate_mini} can be estimated exactly as in Theorem~\ref{thm:proof_asssumption_P1dc} due to \eqref{eq:mini_I2_property} by a element-wise Helmholtz decomposition such that
\begin{align*}
  \int_\Omega (1 - I) \vecb{v} \cdot \ClemN \vecb{\theta} \dx
  \lesssim \| \nabla \vecb{v} \|_{L^2} \| h_\mathcal{T}^2 \mathrm{curl} \vecb{\theta} \|_{L^2}.
\end{align*}
For the second integral, first note that due to $\ClemN \vecb{\theta} \in  H(\mathrm{curl},\omega_y)$ we can use the regular decomposition of Lemma \ref{lem:reg_decomposition} on each patch to get 
\begin{align*}
    \ClemN \vecb{\theta}|_{\omega_y} = \nabla \alpha_y + \vecb{\beta}_y \quad \text{for all } y \in \mathcal{N}.
\end{align*}
with some $\alpha_y \in H^1(\omega_y)$ and $\vecb{\beta}_y \in [H^1(\omega_y)]^2$ such that \(\int_{\omega_y} \vecb{\beta}_y \dx = 0\) and
$$\| \vecb{\beta}_y \|_{H^1(\omega_y)} \lesssim \| \mathrm{curl} (\Pi_{ND_0} \vecb{\theta}) \|_{L^2(\omega_y)} \lesssim \| \mathrm{curl} \vecb{\theta} \|_{L^2(\omega_y)}.$$
Next note, that on each element $T \subset \omega_y$ we have
$\Pi_{ND_0} \vecb{\theta}|_T \in [H^1(T)]^2$ and thus 
\begin{align*}
  \nabla \alpha_y|_T = \Pi_{ND_0} \vecb{\theta}|_T - \vecb{\beta_y}|_T \in [H^1(T)]^2 \quad \Rightarrow \quad  \alpha_y|_T \in H^2(T).
\end{align*}
Together with the projection property of \(\IN\), the commuting properties \eqref{eqn:commuting_props} of the de Rham complex and
the continuity of the nodal interpolation $I_\mathcal{N}$ for $H^2$ functions, the Helmholtz decomposition can be cast into the discrete version
\begin{align*}
    \Pi_{ND_0} \vecb{\theta}|_{\omega_y} = \IN(\nabla \alpha_y + \vecb{\beta}_y) = \nabla (\IL \alpha_y) + \IN \vecb{\beta}_y.
\end{align*}
Finally, a scaling argument and a Poincar\'e inequality shows
\begin{align*}
  \| \IN \vecb{\beta}_y \|_{L^2(\omega_y)} \lesssim \| \vecb{\beta}_y \|_{L^2(\omega_y)} + h_y \| \nabla \vecb{\beta}_y \|_{L^2(\omega_y)}
  \lesssim h_y \| \nabla \vecb{\beta}_y \|_{L^2(\omega_y)} \lesssim h_y \| \mathrm{curl} \vecb{\theta} \|_{L^2(\omega_y)}.
\end{align*}
Furthermore, note that the reconstruction operator is orthogonal on gradients of continuous \(P_1\)-functions like \(\nabla (\IL \alpha_y)\)
due to \cite[Proposition 16.ii]{2016arXiv160903701L}, i.e. \(\int_{\omega_y} \sigma_y \cdot \nabla (\IL \alpha_y) \dx = 0\).
Now, the second integral of \eqref{eqn:intemediate_mini} is bounded by
\begin{align*}
  \sum_{y \in \mathcal{N}} \int_{\omega_y} h_y^{-1} \sigma_y \cdot h_y \ClemN \vecb{\theta} \dx
  & \lesssim \sum_{y \in \mathcal{N}} \|  h_y^{-1} \sigma_y\|_{L^2(\omega_y)} \| h_y^2 \nabla \vecb{\beta}_y \|_{L^2(\omega_y)}\\
  & \lesssim \sum_{y \in \mathcal{N}} \| \nabla \vecb{v} \|_{L^2} \| h_y^2 \mathrm{curl} \vecb{\theta} \|_{L^2(\omega_y)} \lesssim \| \nabla \vecb{v} \|_{L^2} \| h_\mathcal{T}^2 \mathrm{curl} \vecb{\theta} \|_{L^2}.
\end{align*}
The combination of all previous results concludes the proof. 
\end{proof}

\begin{theorem}[Proof of Assumption~\ref{assumption1} for the Taylor--Hood finite element method]
  The Taylor--Hood finite element method family with the reconstruction operator from \cite{2016arXiv160903701L} and the Fortin operator \(I\) from \cite{Mardal2013,MR3272546}
  in two dimensions with the property
  \begin{align}\label{eq:TH_I2_property}
     \int_\Omega (1 - I) \vecb{v} \cdot \vecb{w} \ds = 0 \quad \text{for all } \vecb{w} \in \widetilde{\mathcal{N}_0}(\mathcal{T}),
  \end{align}
  where \(\widetilde{\mathcal{N}_0}(\mathcal{T})\) is a subset of \(\mathcal{N}_0(\mathcal{T})\) as defined in \cite{Mardal2013,MR3272546},
  satisfy Assumption~\ref{assumption1}.
\end{theorem}
\begin{remark}
  The proof requires some assumption on the mesh, i.e. we require
that each interior face \(E \in E^\circ\)
has at most one node on the boundary \(\partial \Omega\).
This assumption was also needed in \cite{Mardal2013} for the construction of a stable the Fortin interpolator and was later removed in \cite{MR3272546}. Maybe similar arguments can be used in our case.

\end{remark}
\begin{proof}
  A triangle inequality, properties of \(\Pi\), and the \(H^1\)-stability of \(I\) show
  \begin{align*}
     \| (1 - \Pi I) \vecb{v} \|_{L^2(T)} & \leq  \| (1 - \Pi) I\vecb{v} \|_{L^2(T)} + \| (1 - I) \vecb{v} \|_{L^2(T)} \lesssim h_T \| \nabla \vecb{v}\|_{L^2(T)}.
  \end{align*}
  Again using the decomposition from Lemma~\ref{lem:decomp_curlestimate} and the orthogonality
  between gradients and \((1 - \Pi I) \vecb{v}\) leads to
  %
 \begin{align*}
   \int_\Omega (1 - \Pi I) \vecb{v} \cdot \vecb{\theta} \dx &=  \int_\Omega (1 - \Pi I) \vecb{v} \cdot \vecb{y} \dx + \int_\Omega (1 - \Pi I) \vecb{v} \cdot \ClemN \vecb{\theta} \dx
  \end{align*}
  The first integral can be estimated similarly as in the proof of Theorem~\ref{thm:proof_asssumption_mini}. For the second integral we use \((1 - \Pi) I \vecb{v} = \sum_{y \in \mathcal{N}} \vecb{\sigma}_y \) to get
  \begin{align*}
\int_\Omega (1 - \Pi I) \vecb{v} \cdot \ClemN \vecb{\theta} \dx =  \int_\Omega (1 - I) \vecb{v} \cdot \ClemN \vecb{\theta} \dx + \sum_{y \in \mathcal{N}} \int_{\omega_y} h_y^{-1}\vecb{\sigma}_y   \cdot h_y \ClemN \vecb{\theta} \dx.
  \end{align*}
Similarly as in the proof of Theorem~\ref{thm:proof_asssumption_mini} we bound the first term. However the integral (using the orthogonality \eqref{eq:TH_I2_property})
  \begin{align*}
     \int_\Omega (1 - I) \vecb{v} \cdot \ClemN \vecb{\theta} \dx = \int_\Omega (1 - I) \vecb{v} \cdot (1-I_{\widetilde{\mathcal{N}_0}})\ClemN \vecb{\theta} \dx
  \end{align*}
needs a different treatment. To estimate this integral we have to design a proper interpolation \(I_{\widetilde{\mathcal{N}_0}}(\ClemN \vecb{\theta})\)
of \(\ClemN \vecb{\theta}\) into the space \(\widetilde{\mathcal{N}_0}(\mathcal{T})\). To do so, we can write \(\ClemN \vecb{\theta}\) as a linear combination
\begin{align*}
   \ClemN \vecb{\theta} = \sum_{E \in \mathcal{E}} \alpha_E N_E \quad \text{with coefficients } \alpha_E := \int_E \ClemN \vecb{\theta} \cdot \vecb{\tau}_E \ds
\end{align*}
and N\'ed\'elec basis functions $N_E$ with \(\int_F N_E \vecb{\tau}_F \ds = \delta_{EF}\) for \(E,F \in \mathcal{E}\).
Then, we choose \(I_{\widetilde{\mathcal{N}_0}}(\ClemN \vecb{\theta})\) as
\begin{align*}
   I_{\widetilde{\mathcal{N}_0}}(\ClemN \vecb{\theta}) := \sum_{E \in \mathcal{E}^0} \alpha_E \widetilde{N}_E
\end{align*}
where \(\mathcal{E}^0\) are the interior edges and \(\widetilde{N}_E\) are the modified basis functions as in \cite{MR3272546}, i.e.
\(\widetilde{N}_E = N_E\) for all edges \(E\) with two interior endpoints and \(\widetilde{N}_E = N_E \pm N_F\) for interior edges \(E\) with
one boundary endpoint and \(F\) is a boundary edge with the same boundary endpoint and in the same triangle of \(E\). The sign depends on the orientation of the tangent vectors. Assume a boundary triangle \(T_E\) with
nodes \(1,2,3\), boundary edge \(E_3 = \mathrm{conv} \lbrace 1, 2 \rbrace\) and two adjacent interior edges \(E_1\) and \(E_2\) as depicted in Figure~\ref{fig:numbering}.
We further assume, that the tangential vectors are pointing from the lower to the larger node number. Then, according to \cite{MR3272546}, the modified basis functions read
\(\widetilde{N}_{E_2} = N_{E_2} + N_{E_3}\) and \(\widetilde{N}_{E_1} = N_{E_1} - N_{E_3}\). Hence, locally on \(T\), we have
\begin{align*}
  \left((1-I_{\widetilde{\mathcal{N}_0}})\ClemN \vecb{\theta}\right)|_T
  & = \alpha_{E_1} N_{E_1} + \alpha_{E_2} N_{E_2} + \alpha_{E_3} N_{E_3} - (\alpha_{E_1} \widetilde{N}_{E_1} + \alpha_{E_2} \widetilde{N}_{E_2})\\
  & = N_{E_3} (\alpha_{E_3} + \alpha_{E_1} - \alpha_{E_2}).
\end{align*}

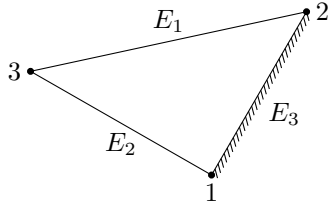
\begin{figure}[!tb]
\begin{center}
\begin{tikzpicture}[scale=0.25]
\path (0,0) coordinate (P1)
      (60:10) coordinate (P2)
      (150:11) coordinate (P3)
      (62:7.5) coordinate (CO)
      (intersection of P2--P3 and P1--CO) coordinate (C)
      (P1)--(P2) coordinate[midway] (Q3)
      (P2)--(P3) coordinate[midway] (Q1)
      (P3)--(P1) coordinate[midway] (Q2)
      (intersection of P1--Q1 and P2--Q2) coordinate (M)
      (M)--(P1) coordinate[midway] (M1)
      (M)--(P2) coordinate[midway] (M2)
      (M)--(P3) coordinate[midway] (M3)
      (Q1)--(P2) coordinate[midway] (M4)
      (M)--(P3) coordinate[midway] (T);
\draw (P1)--(P2)--(P3)--(P1);
\draw[decorate,decoration={border,segment length=0.8mm, amplitude=1.25mm, radius=50mm}, thin] (P2)--(P1);
\foreach \x in {P1,P2,P3}
{%
\filldraw[fill=white,very thick] (\x) circle (0.1);
}
\node at (P1)[below]{$1$};
\node at (P2)[right]{$2$};
\node at (P3)[left]{$3$};
\node at (Q1)[above]{$E_{1}$};
\node at (Q2)[below]{$E_{2}$};
\node at (Q3)[below right]{$E_{3}$};
\end{tikzpicture}
\vspace{0.25cm}
\caption{
Enumeration of the vertices and edges in a boundary triangle with boundary edge \(E_3\).}
\label{fig:numbering}
\end{center}
\end{figure}

The definition of \(\alpha_{E_j}\) and an easy calculation plus the Stokes theorem show
\begin{align*}
\alpha_{E_3} + \alpha_{E_1} - \alpha_{E_2} = \int_{\partial T} \ClemN \vecb{\theta} \cdot \vecb{\tau} \ds = \int_{T} \mathrm{curl} (\ClemN \vecb{\theta}) \dx
\end{align*}
and hence the estimate
\begin{align*}
  \| (1-I_{\widetilde{\mathcal{N}_0}})\ClemN \vecb{\theta} \|_{L^2(T)}
  \leq \left\lvert \int_{T} \mathrm{curl} (\ClemN \vecb{\theta}) \dx \right\rvert \| N_E \|_{L^2(T)}
  \lesssim h_T \| \mathrm{curl} (\ClemN \vecb{\theta}) \|_{L^2(T)}.
\end{align*}
On interior triangles, it holds \(\ClemN \vecb{\theta} - I_{\widetilde{\mathcal{N}_0}}(\ClemN \vecb{\theta}) = 0\) and hence
\begin{align*}
     \int_\Omega (1 - I) \vecb{v} \cdot & \ClemN \vecb{\theta} \dx
      = \int_\Omega (1 - I) \vecb{v} \cdot (\ClemN \vecb{\theta} - I_{\widetilde{\mathcal{N}_0}}(\ClemN \vecb{\theta})) \dx\\
     & \lesssim \sum_{T \in \mathcal{T}(\partial \Omega)} h_T^2 \| \mathrm{curl} (\ClemN \vecb{\theta}) \|_{L^2(T)} \| \nabla \vecb{v} \|_{L^2(\omega_T)}
      \lesssim \| h_{\mathcal{T}}^2 \mathrm{curl} \vecb{\theta} \|_{L^2} \| \nabla \vecb{v} \|_{L^2}.
  \end{align*}

This concludes the proof. 

\end{proof}

\section{Numerical experiments}\label{sec:numerics}
In the following two numerical examples, the novel error estimator
\begin{align*}
  \mu_\text{new}^2 := \nu^{-2} \eta_\text{new}(\nabla \vecb{u}_h)^2 + \| \mathrm{div} \vecb{u}_h \|^2_{L^2}
\end{align*}
from Theorem~\ref{thm:eta_averaging_pr} (for pressure-robust methods) or Proposition~\ref{thm:eta_averaging_cl} (for classical methods)
is compared to the classical error estimator
\begin{align*}
  \mu_\text{class}^2 :=  \nu^{-2} \eta_\text{class}(\nabla u_h,p_h)^2
  + \| \mathrm{div} \vecb{u}_h \|^2_{L^2}
\end{align*}
from Theorem~\ref{thm:errorbounds}, with respect to the $H^1$-seminorm $\textrm{err}_{H^1}(u_h) := || \nabla u - \nabla u_h ||_{L^2}$.
Our adaptive mesh refinement algorithm follows the loop
\begin{align*}
   \mathrm{SOLVE} \rightarrow \mathrm{ESTIMATE} \rightarrow \mathrm{MARK} \rightarrow \mathrm{REFINE} \rightarrow \mathrm{SOLVE} \rightarrow \ldots
\end{align*}
and employs the local contributions to the error estimator as
element-wise refinement indicators.
In the marking step, an element \(T \in \mathcal{T}\) is marked for refinement if \(\mu(T) \geq \frac{1}{4} \max\limits_{K \in \mathcal{T}} \mu(K)\).
The refinement step refines all marked elements plus further elements in a closure step to guarantee a regular triangulation. The implementation and numerical examples where performed with NGSolve/Netgen \cite{ngsolve}, \cite{netgen}.

\begin{remark}
For reducing the costs of the estimator, we estimated the consistency error $\eta_\text{cons,1}(\nabla u_h) = \| \nu \mathrm{div}_h (\nabla u_h) \circ (1 - \Pi) \|_{V_{0,h}^\star}$ according to
\eqref{eqn:reconstop_local_orthogonalities} by
\begin{align*}
  \eta_\text{cons,1}(\nabla u_h) \lesssim \nu \left( \sum_{K \in \mathcal{K}}
  h_K^2 \| (1 - \pi_{P_{q-1}(K)}) \Delta_h u_h \|^2_{L^2(K)} \right)^{1/2}.
\end{align*}
\end{remark}

\subsection{Smooth example on unit square} \label{curltwodimexample}
This example concerns the Stokes problem for
\begin{align*}
  \vecb{u}(x,y) := \mathrm{curl} \left(x^2(x-1)^2y^2(y-1)^2\right) \quad \text{and} \quad  p(x,y) := x^5 + y^5 - 1/3
\end{align*}
on the unit square \(\Omega := (0,1)^2\) with matching right-hand side \(\vecb{f} := - \nu \Delta \vecb{u} + \nabla p\)
for variable viscosity \(\nu\).

\begin{table}
  \begin{center}
    \caption{\label{tab:CurlExampleFixedmesh}The $H^1$ error and the old and new error estimators including the efficiency for the example of section \ref{curltwodimexample} for varying $\nu$ using the classical Taylor Hood element $\text{TH}_2$ and its pressure robust modification.}
  \footnotesize
  \begin{tabular}{c@{~~} @{~~}c@{~~} @{~~}c@{~~}  @{~~}c@{~~} @{~~}c@{~~} @{~~}c@{~~} @{~~}c@{~~} @{~~}c@{~~}}
    \toprule
     & \multicolumn{3}{c}{ (classical)} &\multicolumn{3}{c}{(p-robust)}\\
        $\nu$ & $ \textrm{err}_{H^1}(u_h)$ & $\mu_\text{class}$ & $\frac{\mu_\text{class}}{\textrm{err}_{H^1}(u_h)} $
& $\textrm{err}_{H^1}(u_h)$ & $\mu_\text{new}$ & $\frac{\mu_\text{new}}{\textrm{err}_{H^1}(u_h)} $\\
    \midrule
    $10^{1}$& \numcoef{0.001265847525399444}& \numcoef{0.01994406784451481}& \numcoef{15.755505654775735}& \numcoef{0.001303824956633806}& \numcoef{0.05192592441325167}& \numcoef{39.82584023189215}\\
$10^{0}$& \numcoef{0.001297267918076333}& \numcoef{0.01416918684944461}& \numcoef{10.922328882113677}& \numcoef{0.0013038249566338068}& \numcoef{0.034652174672599206}& \numcoef{26.57732120887117}\\
$10^{-1}$& \numcoef{0.0031200912873200794}& \numcoef{0.11172402531285176}& \numcoef{35.80793477642578}& \numcoef{0.0013038249566338315}& \numcoef{0.032924996269452764}& \numcoef{25.252620071376253}\\
$10^{-2}$& \numcoef{0.028546945385180832}& \numcoef{1.1123820930363155}& \numcoef{38.96676432546687}& \numcoef{0.0013038249566340158}& \numcoef{0.03275228154554611}& \numcoef{25.12015234782754}\\
$10^{-3}$& \numcoef{0.28519134950247516}& \numcoef{11.121735407091114}& \numcoef{38.99745005061799}& \numcoef{0.0013038249566365867}& \numcoef{0.03273501010618053}& \numcoef{25.106905600752906}\\
$10^{-4}$& \numcoef{2.851885435299173}& \numcoef{111.21554661028321}& \numcoef{38.99719996943577}& \numcoef{0.0013038249566595386}& \numcoef{0.032733282963005116}& \numcoef{25.105580926192225}\\
$10^{-5}$& \numcoef{28.518851311410526}& \numcoef{1112.1536864539069}& \numcoef{38.997141726003846}& \numcoef{0.0013038249568637459}& \numcoef{0.03273311025353761}& \numcoef{25.105448458568148}\\
$10^{-6}$& \numcoef{285.1885125743553}& \numcoef{11121.535087671366}& \numcoef{38.997135569307765}& \numcoef{0.0013038249590302574}& \numcoef{0.03273309302043661}& \numcoef{25.105435199509003}\\
    \bottomrule
  \end{tabular}  
  \end{center}
\end{table}

Table~\ref{tab:CurlExampleFixedmesh} lists the error of the classical and pressure-robust Taylor-Hood finite element methods with their error estimators \(\mu_\text{class}\) and \(\mu_\text{new}\)
on a fixed mesh with 1139 degrees of freedom but varying viscosities \(\nu \in (10^{-6},10)\).
As expected by the a priori error estimates of Theorems~\ref{thm:apriori_classical} and Theorem~\ref{thm:apriori_probust},
the error of the classical solution scales with \(\nu^{-1}\), while the error of the pressure-robust method is \(\nu\)-invariant.
Another observation is that both error estimators are efficient for their respective discrete solution.

Figure~\ref{curltwodexample} compares the errors and error estimators of the
 Taylor--Hood finite element method of order $2$ and the MINI finite element method
with and without the pressure robust modification for uniform mesh refinement
as in the case \(\nu = 1\) and a pressure-dominant case with \(\nu = 10^{-3}\).

In the pressure dominant case $\nu = 10^{-3}$ the right hand side $\vecb{f}$ tends to have a large irrotational part.
The left plot of Figure~\ref{curltwodexample} confirms once again that the velocity error scales with
$1 / \nu$ and that pressure-robust methods result in much more accurate solutions.
For the classical methods both estimators $\mu_\text{new}$ and $\mu_\text{class}$ are efficient,
i.e.\ have comparable overestimation factors and the same optimal convergence order as the velocity error.
In case of the MINI finite element method, all quantities even converge quadratically.
This is due to the dominance of the pressure error and the higher approximation order of the pressure.
In this sense, we are in a pre-asymptotic range and the error will convergence linearly as soon as the
\(\nu^{-3}\)-weighted pressure error is of same magnitude (as it is the case for $\nu=1$ from the very beginning).
Also for the classical MINI element $\mu_\text{new}$ and $\mu_\text{class}$ are efficient with a comparable overestimation factor.

For the pressure-robust methods we observe that for both elements the novel estimator $\mu_\text{new}$ is much smaller than $\mu_\text{class}$. To be more precise, it scales with $\mu_\text{new} \approx 1/ \nu ~ \mu_\text{class}$ in case of the Taylor-Hood method as expected by the theory.
    This is again due to the discrete pressure that is used in $\mu_\text{class}$ ($p_h$ replaced by some better approximation of $p$ would reduce the gap between $\mu_\text{new}$ and $\mu_\text{class}$).
    Hence, $\mu_\text{new}$ is efficient and $\mu_\text{class}$ is not efficient for the pressure-robust Taylor--Hood finite element method.
    In case of the pressure-robustly modified MINI method, the velocity error and the novel estimator $\mu_\text{new}$ now have the expected optimal linear order of the MINI finite element method.
    Otherwise, the conclusions are similar to the ones for the Taylor--Hood method.
    
In this case \(\nu = 1\) the irrotational part and the rotational part of the right hand side $\vecb{f}$ have the same magnitude, thus the pressure error has not
such a big impact on the accuracy of the discrete velocity.
Accordingly, there is only little to no improvement by the application of the pressure-robust modification.
Thus, in the right plots of Figure~\ref{curltwodexample} we can see that the velocity error of both methods, the pressure robust and the classical one, is of the same magnitude and order. Both estimators are efficient with slightly less overestimation by $\eta_\text{class}$.

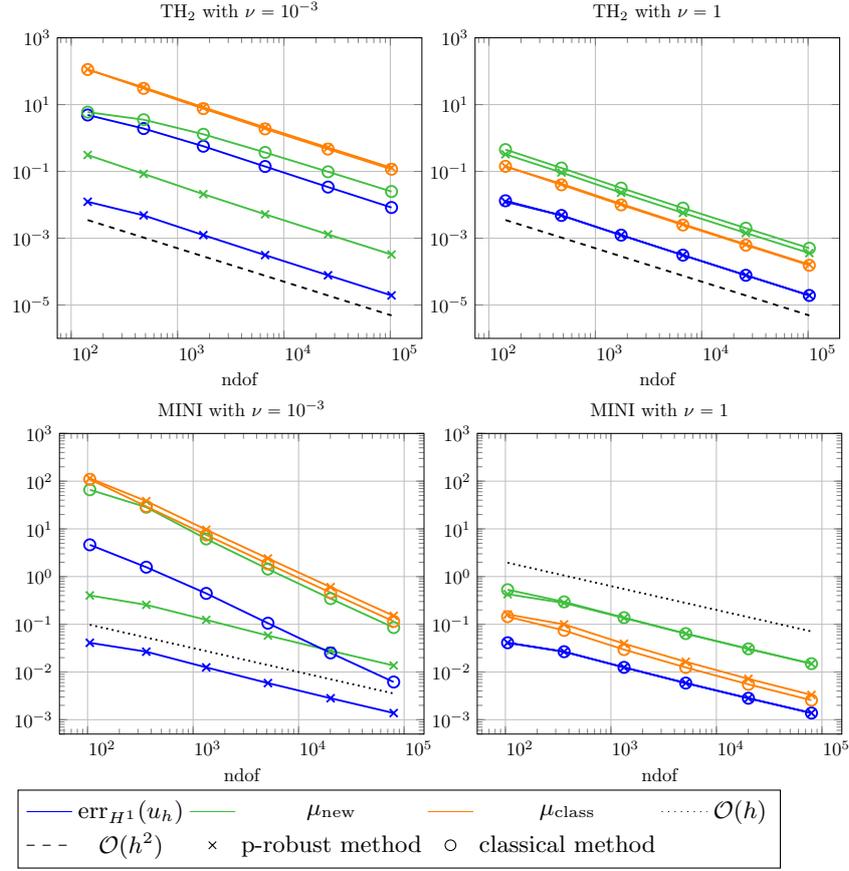
\begin{figure}
  \begin{center}
  \pgfplotstableread{plots/data/curlu_0.001_probust_TH2.out} \THtworecsmall
\pgfplotstableread{plots/data/curlu_1_probust_TH2.out} \THtworecbig
\pgfplotstableread{plots/data/curlu_0.001_TH2.out} \THtwosmall
\pgfplotstableread{plots/data/curlu_1_TH2.out} \THtwobig

\pgfplotstableread{plots/data/curlu_0.001_probust_MINI.out} \MINIrecsmall
\pgfplotstableread{plots/data/curlu_1_probust_MINI.out} \MINIrecbig
\pgfplotstableread{plots/data/curlu_0.001_MINI.out} \MINIsmall
\pgfplotstableread{plots/data/curlu_1_MINI.out} \MINIbig

\begin{tikzpicture}
  [
  scale=0.7
  ]
  \begin{groupplot}[
    group style ={group size = 2 by 1},
    name=plot1,
    title={The TH2 element with $\nu = 10^{-3}$ (left) and $\nu = 1$ (right)},
    scale=1.0,
    xlabel=ndof,
    legend columns=1,
    legend style={font=\small},
    ylabel=$$,
    xmode=log,
    ymax=1e3,
    ymin=1e-6,
    ymode=log,
    xticklabel style={text width=3em,align=right},
    yticklabel style={text width=3em,align=right},
    %
    grid=major,
    ]
        
    \nextgroupplot[title = {$\text{TH}_2$ with $\nu = 10^{-3}$}, legend to name={mylegend}]    

    \addplot[line width=1pt, dashed, color=black] table[x=0,y expr={\thisrowno{1}*\thisrowno{1}*0.5 } ]{\THtworecsmall};
    \addplot[line width=1pt, mark=x, mark size=3, color=blue] table[x=0,y=2]{\THtworecsmall};
    \addplot[line width=1pt, mark=x, mark size=3, color=green!50!gray] table[x=0,y=3]{\THtworecsmall};
    \addplot[line width=1pt, mark=x, mark size=3, color=orange] table[x=0,y=4]{\THtworecsmall};

    \addplot[line width=1pt, mark=o, mark size=3, color=blue] table[x=0,y=2]{\THtwosmall};
    \addplot[line width=1pt, mark=o, mark size=3, color=green!50!gray] table[x=0,y=3]{\THtwosmall};
    \addplot[line width=1pt, mark=o, mark size=3, color=orange] table[x=0,y=4]{\THtwosmall};

    \nextgroupplot[title = {$\text{TH}_2$ with $\nu = 1$}, legend to name = dummy]
    \addplot[line width=1pt, dashed, color=black] table[x=0,y expr={\thisrowno{1}*\thisrowno{1}*0.5} ]{\THtworecbig}; 
    \addplot[line width=1pt, mark=x, mark size=3, color=blue] table[x=0,y=2]{\THtworecbig}; 
    \addplot[line width=1pt, mark=x, mark size=3, color=green!50!gray] table[x=0,y=3]{\THtworecbig};
    \addplot[line width=1pt, mark=x, mark size=3, color=orange] table[x=0,y=4]{\THtworecbig};

    \addplot[line width=1pt, mark=o, mark size=3, color=blue] table[x=0,y=2]{\THtwobig};
    \addplot[line width=1pt, mark=o, mark size=3, color=green!50!gray] table[x=0,y=3]{\THtwobig};
    \addplot[line width=1pt, mark=o, mark size=3, color=orange] table[x=0,y=4]{\THtwobig};

\end{groupplot}


\end{tikzpicture}

\begin{tikzpicture}
  [
  scale=0.7
  ]
  \begin{groupplot}[
    group style ={group size = 2 by 1},
    name=plot1,
    title={The MINI element with $\nu = 10^{-3}$ (left) and $\nu = 1$ (right)},
    scale=1.0,
    xlabel=ndof,
    legend columns=4,
    legend style={font=\small},
    ylabel=$$,
    xmode=log,
    ymax=1e3,
    ymin=5e-4,
    ymode=log,
    xticklabel style={text width=3em,align=right},
    yticklabel style={text width=4em,align=right},
    %
     grid=major,
    ]
        
    \nextgroupplot[title = {MINI with $\nu = 10^{-3}$}, legend to name={mylegend}]    
    \addlegendentry{$\textrm{err}_{H^1}(u_h)$ }
    \addlegendimage{no markers, blue}
    \addlegendentry{$\mu_\text{new}$  }
    \addlegendimage{no markers, green!50!gray}
    \addlegendentry{$\mu_\text{class}$ }
    \addlegendimage{no markers, orange}
    \addlegendentry{$\mathcal{O}(h)$}
    \addlegendimage{dotted, black}
    \addlegendentry{$\mathcal{O}(h^2)$}
    \addlegendimage{dashed, black}
    \addlegendentry{p-robust method}
    \addlegendimage{only marks, mark = x}
    \addlegendentry{classical method}
    \addlegendimage{only marks, mark = o}

    \addplot[line width=1pt, dotted, color=black] table[x=0,y expr={\thisrowno{1} } ]{\MINIrecsmall};
    \addplot[line width=1pt, mark=x, mark size=3, color=blue] table[x=0,y=2]{\MINIrecsmall};
    \addplot[line width=1pt, mark=x, mark size=3, color=green!50!gray] table[x=0,y=3]{\MINIrecsmall};
    \addplot[line width=1pt, mark=x, mark size=3, color=orange] table[x=0,y=4]{\MINIrecsmall};

    \addplot[line width=1pt, mark=o, mark size=3, color=blue] table[x=0,y=2]{\MINIsmall};
    \addplot[line width=1pt, mark=o, mark size=3, color=green!50!gray] table[x=0,y=3]{\MINIsmall};
    \addplot[line width=1pt, mark=o, mark size=3, color=orange] table[x=0,y=4]{\MINIsmall};

    \nextgroupplot[title = {MINI with $\nu = 1$}, legend to name = dummy]
    \addplot[line width=1pt, dotted, color=black] table[x=0,y expr={\thisrowno{1}*20 } ]{\MINIrecbig}; 
    \addplot[line width=1pt, mark=x, mark size=3, color=blue] table[x=0,y=2]{\MINIrecbig}; 
    \addplot[line width=1pt, mark=x, mark size=3, color=green!50!gray] table[x=0,y=3]{\MINIrecbig};
    \addplot[line width=1pt, mark=x, mark size=3, color=orange] table[x=0,y=4]{\MINIrecbig};

    \addplot[line width=1pt, mark=o, mark size=3, color=blue] table[x=0,y=2]{\MINIbig};
    \addplot[line width=1pt, mark=o, mark size=3, color=green!50!gray] table[x=0,y=3]{\MINIbig};
    \addplot[line width=1pt, mark=o, mark size=3, color=orange] table[x=0,y=4]{\MINIbig};

\end{groupplot}

\node[right=1em,inner sep=0pt] at (-1.3,-1.8) {\pgfplotslegendfromname{mylegend}};

\end{tikzpicture}

  \caption{The $H^1$-error, $\mu_\text{class}$ and $\mu_\text{new}$ for the example of section \ref{curltwodimexample} with $\nu = 1$ and $\nu = 10^{-3}$. At the top the $\text{TH}_2$ element, and at the bottom the MINI element.} \label{curltwodexample}
    \end{center}
\end{figure}

%

\begin{figure}
  \begin{center}
  \pgfplotstableread{plots/data/lshape_0.001_adaptive_probust_P20.out} \lshapeadaprecptz
\pgfplotstableread{plots/data/lshape_0.001_adaptive_P20.out} \lshapeadapptz
\pgfplotstableread{plots/data/lshape_0.001_probust_P20.out} \lshaperecptz
\pgfplotstableread{plots/data/lshape_0.001_P20.out} \lshapeptz

\pgfplotstableread{plots/data/lshape_0.001_adaptive_probust_P2B.out} \lshapeadaprecptb
\pgfplotstableread{plots/data/lshape_0.001_adaptive_P2B.out} \lshapeadapptb
\pgfplotstableread{plots/data/lshape_0.001_probust_P2B.out} \lshaperecptb
\pgfplotstableread{plots/data/lshape_0.001_P2B.out} \lshapeptb
                         
\begin{tikzpicture}
  [
  scale=0.7
  ]
  \begin{groupplot}[
    group style ={group size = 2 by 1},
    name=plot2,
    title={pressure robust method},
    scale=1.0,
    xlabel=ndof,
    legend columns=1,
    ylabel=$$,
    xmin = 1e2,
    xmax=2e5,
    xmode=log,
    ymax=1e3,
    ymin=2e-2,
    ymode=log,
    yticklabel style={text width=3em,align=right},
    xticklabel style={text width=3em,align=right},
    %
     grid=major,
    ]
        
    \nextgroupplot[title = P2P0 (p-robust), legend to name={mylegend}]
    \addlegendentry{$\mathcal{O}(h)$}
    \addlegendimage{dotted, black}
    \addlegendentry{$\|\nabla(\boldsymbol{u}-\boldsymbol{u}_h) \|_{L^2}$ \hspace{10pt}}
    \addlegendimage{no markers, blue}
    \addlegendentry{$\mu_\text{new}$ \hspace{10pt}}
    \addlegendimage{no markers, green!50!gray}
    \addlegendentry{adaptive ref. \hspace{10pt}}
    \addlegendimage{only marks, mark = x}
    \addlegendentry{uniform ref.}
    \addlegendimage{only marks, mark = o}

    \addplot[line width=1pt, dotted, color=black] table[x=0,y expr={\thisrowno{1}*10} ]{\lshapeadaprecptz}; 
    \addplot[line width=1pt, mark=x, color=blue] table[x=0,y=2]{\lshapeadaprecptz}; 
    \addplot[line width=1pt, mark=x, color=green!50!gray] table[x=0,y=3]{\lshapeadaprecptz};

    \addplot[line width=1pt, mark=o, color=blue] table[x=0,y=2]{\lshaperecptz};
    \addplot[line width=1pt, mark=o, color=green!50!gray] table[x=0,y=3]{\lshaperecptz};

    \nextgroupplot[title = P2P0 (classical), legend to name = dummy]
    \addplot[line width=1pt, dotted, color=black] table[x=0,y expr={\thisrowno{1}*2000} ]{\lshapeadaprecptz};
    \addplot[line width=1pt, mark=x, color=blue] table[x=0,y=2]{\lshapeadapptz};
    \addplot[line width=1pt, mark=x, color=green!50!gray] table[x=0,y=3]{\lshapeadapptz};

    \addplot[line width=1pt, mark=o, color=blue] table[x=0,y=2]{\lshapeptz};
    \addplot[line width=1pt, mark=o, color=green!50!gray] table[x=0,y=3]{\lshapeptz};
\end{groupplot}


\end{tikzpicture}
\begin{tikzpicture}
  [
  scale=0.7
  ]
  \begin{groupplot}[
    group style ={group size = 2 by 1},
    name=plot2,
    title={pressure robust method},
    scale=1.0,
    xlabel=ndof,
    legend columns=3,
    ylabel=$$,
    xmode=log,
    xmin = 1e2,
    xmax=2e5,
    ymax=5e2,
    ymin=1e-3,
    ymode=log,
    yticklabel style={text width=3em,align=right},
    xticklabel style={text width=3em,align=right},
    %
     grid=major,
    ]
        
    \nextgroupplot[title = P2B (p-robust), legend to name={mylegend}]
    \addlegendentry{$\textrm{err}_{H^1}(u_h)$ \hspace{10pt}}
    \addlegendimage{no markers, blue}    
    \addlegendentry{$\mathcal{O}(h)$}
    \addlegendimage{dotted, black}
    \addlegendentry{$\mathcal{O}(h^2)$}
    \addlegendimage{dashed, black}
    \addlegendentry{$\mu_\text{new}$ \hspace{10pt}}
    \addlegendimage{no markers, green!50!gray}    
    \addlegendentry{adaptive ref. \hspace{10pt}}
    \addlegendimage{only marks, mark = x}
    \addlegendentry{uniform ref.}
    \addlegendimage{only marks, mark = o}

    \addplot[line width=1pt, dashed, color=black] table[x=0,y expr={\thisrowno{1}^2*100} ]{\lshapeadaprecptb}; 
    \addplot[line width=1pt, mark=x, color=blue] table[x=0,y=2]{\lshapeadaprecptb}; 
    \addplot[line width=1pt, mark=x, color=green!50!gray] table[x=0,y=3]{\lshapeadaprecptb};

    \addplot[line width=1pt, mark=o, color=blue] table[x=0,y=2]{\lshaperecptb};
    \addplot[line width=1pt, mark=o, color=green!50!gray] table[x=0,y=3]{\lshaperecptb};

    \nextgroupplot[title = P2B (classical), legend to name = dummy]
    \addplot[line width=1pt, dashed, color=black] table[x=0,y expr={\thisrowno{1}^2*1000} ]{\lshapeadaprecptb};
    \addplot[line width=1pt, mark=x, color=blue] table[x=0,y=2]{\lshapeadapptb};
    \addplot[line width=1pt, mark=x, color=green!50!gray] table[x=0,y=3]{\lshapeadapptb};

    \addplot[line width=1pt, mark=o, color=blue] table[x=0,y=2]{\lshapeptb};
    \addplot[line width=1pt, mark=o, color=green!50!gray] table[x=0,y=3]{\lshapeptb};
\end{groupplot}

\node[right=1em,inner sep=0pt] at (0.5,-1.8) {\pgfplotslegendfromname{mylegend}};

\end{tikzpicture}

  \vspace{5pt}
  \caption{Error for L-shape example of section \ref{lshapeexample} using the discontinuous pressure elements P2P0 (top) and the P2B (bottom) with $\nu=10^{-3}$} \label{fig:lshape:dcontpres}
  \end{center}
\end{figure}
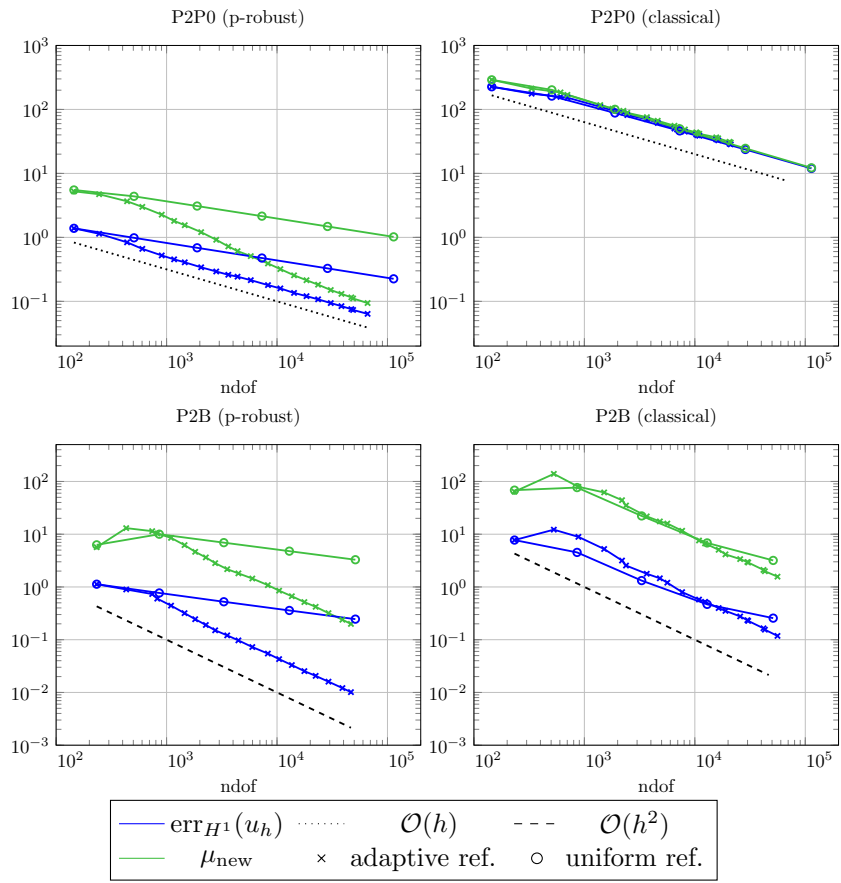

\begin{figure}
  \begin{center}
  \pgfplotstableread{plots/data/lshape_0.001_adaptive_probust_TH3.out} \lshapeadaprectht
\pgfplotstableread{plots/data/lshape_0.001_adaptive_TH3.out} \lshapeadaptht
\pgfplotstableread{plots/data/lshape_0.001_probust_TH3.out} \lshaperectht
\pgfplotstableread{plots/data/lshape_0.001_TH3.out} \lshapetht

\pgfplotstableread{plots/data/lshape_0.001_adaptive_probust_MINI.out} \lshapeadaprecmini
\pgfplotstableread{plots/data/lshape_0.001_adaptive_MINI.out} \lshapeadapmini
\pgfplotstableread{plots/data/lshape_0.001_probust_MINI.out} \lshaperecmini
\pgfplotstableread{plots/data/lshape_0.001_MINI.out} \lshapemini

\begin{tikzpicture}
  [
  scale=0.7
  ]
  \begin{groupplot}[
    group style ={group size = 2 by 1},
    name=plot2,
    title={pressure robust method},
    scale=1.0,
    xlabel=ndof,
    legend columns=1,
    ylabel=$$,
    xmode=log,
    xmin = 7e1,
    xmax=1e5,
    ymax=1e3,
    ymin=2e-2,
    ymode=log,
    yticklabel style={text width=3em,align=right},
    xticklabel style={text width=3em,align=right},
    %
     grid=major,
    ]
        
    \nextgroupplot[title = MINI (p-robust), legend to name={mylegend}]
    \addlegendentry{$\mathcal{O}(h)$}
    \addlegendimage{dotted, black}
    \addlegendentry{$\|\nabla(\boldsymbol{u}-\boldsymbol{u}_h) \|_{L^2}$ \hspace{10pt}}
    \addlegendimage{no markers, blue}
    \addlegendentry{$\mu_\text{new}$ \hspace{10pt}}
    \addlegendimage{no markers, green!50!gray}
    \addlegendentry{adaptive ref. \hspace{10pt}}
    \addlegendimage{only marks, mark = x}
    \addlegendentry{uniform ref.}
    \addlegendimage{only marks, mark = o}

   \addplot[line width=1pt, dotted, color=black] table[x=0,y expr={\thisrowno{1}*10} ]{\lshapeadaprecmini}; 
    \addplot[line width=1pt, mark=x, color=blue] table[x=0,y=2]{\lshapeadaprecmini}; 
    \addplot[line width=1pt, mark=x, color=green!50!gray] table[x=0,y=3]{\lshapeadaprecmini};

    \addplot[line width=1pt, mark=o, color=blue] table[x=0,y=2]{\lshaperecmini};
    \addplot[line width=1pt, mark=o, color=green!50!gray] table[x=0,y=3]{\lshaperecmini};

    \nextgroupplot[title = MINI (classical), legend to name = dummy]
    \addplot[line width=1pt, dotted, color=black] table[x=0,y expr={\thisrowno{1}*10} ]{\lshapeadaprecmini};
    \addplot[line width=1pt, mark=x, color=blue] table[x=0,y=2]{\lshapeadapmini};
    \addplot[line width=1pt, mark=x, color=green!50!gray] table[x=0,y=3]{\lshapeadapmini};

    \addplot[line width=1pt, mark=o, color=blue] table[x=0,y=2]{\lshapemini};
    \addplot[line width=1pt, mark=o, color=green!50!gray] table[x=0,y=3]{\lshapemini};
\end{groupplot}


\end{tikzpicture}

\begin{tikzpicture}
  [
  scale=0.7
  ]
  \begin{groupplot}[
    group style ={group size = 2 by 1},
    name=plot2,
    title={pressure robust method},
    scale=1.0,
    xlabel=ndof,
    legend columns=3,
    ylabel=$$,
    xmode=log,
    xmin = 7e1,
    xmax=1e5,
    ymax=5e2,
    ymin=2e-5,
    ymode=log,
    yticklabel style={text width=3em,align=right},
    xticklabel style={text width=3em,align=right},
    %
     grid=major,
    ]
        
    \nextgroupplot[title = $\text{TH}_3$ (p-robust), legend to name={mylegend}]
    
    \addlegendentry{$\textrm{err}_{H^1}(u_h)$ \hspace{10pt}}
    \addlegendimage{no markers, blue}    
    \addlegendentry{$\mathcal{O}(h)$}
    \addlegendimage{dotted, black}
    \addlegendentry{$\mathcal{O}(h^3)$}
    \addlegendimage{dashed, black}
    \addlegendentry{$\mu_\text{new}$ \hspace{10pt}}
    \addlegendimage{no markers, green!50!gray}        
    \addlegendentry{adaptive ref. \hspace{10pt}}
    \addlegendimage{only marks, mark = x}
    \addlegendentry{uniform ref.}
    \addlegendimage{only marks, mark = o}

   \addplot[line width=1pt, dashed, color=black] table[x=0,y expr={\thisrowno{1}^3*1000} ]{\lshapeadaprectht}; 
    \addplot[line width=1pt, mark=x, color=blue] table[x=0,y=2]{\lshapeadaprectht}; 
    \addplot[line width=1pt, mark=x, color=green!50!gray] table[x=0,y=3]{\lshapeadaprectht};

    \addplot[line width=1pt, mark=o, color=blue] table[x=0,y=2]{\lshaperectht};
    \addplot[line width=1pt, mark=o, color=green!50!gray] table[x=0,y=3]{\lshaperectht};

    \nextgroupplot[title = $\text{TH}_3$ (classical), legend to name = dummy]
    \addplot[line width=1pt, dashed, color=black] table[x=0,y expr={\thisrowno{1}^3*10000} ]{\lshapeadaprectht};
    \addplot[line width=1pt, mark=x, color=blue] table[x=0,y=2]{\lshapeadaptht};
    \addplot[line width=1pt, mark=x, color=green!50!gray] table[x=0,y=3]{\lshapeadaptht};

    \addplot[line width=1pt, mark=o, color=blue] table[x=0,y=2]{\lshapetht};
    \addplot[line width=1pt, mark=o, color=green!50!gray] table[x=0,y=3]{\lshapetht};
\end{groupplot}

\node[right=1em,inner sep=0pt] at (0.5,-1.8) {\pgfplotslegendfromname{mylegend}};

\end{tikzpicture}

  \vspace{5pt}
  \caption{Error for L-shape example of section \ref{lshapeexample} using the continuous pressure elements MINI (top) and the $\text{TH}_3$ (bottom) with $\nu=10^{-3}$} \label{fig:lshape:contpres}
  \end{center}
\end{figure}

\subsection{L-shape example} \label{lshapeexample}
This example studies a velocity \(\vecb{u}\) and a pressure \(p_0\) on the L-shaped domain 
\(\Omega := (-1,1)^2 \setminus \left((0,1) \times (-1,0)\right)\) taken from \cite{MR993474}
that satisfy \(-\nu \Delta \vecb{u} + \nabla p_0 = 0\). The fields are defined in polar coordinates and read
\begin{align*}
 \vecb{u}(r,\varphi)
& :=r^\alpha
\begin{pmatrix}
(\alpha+1)\sin(\varphi)\psi(\varphi) + \cos(\varphi)\psi^\prime(\varphi)
\\
-(\alpha+1)\cos(\varphi)\psi(\varphi) + \sin(\varphi)\psi^\prime(\varphi)
\end{pmatrix}^T,\\
 p_0 &:= \nu^{-1} r^{(\alpha-1)}((1+\alpha)^2 \psi^\prime(\varphi)+\psi^{\prime\prime\prime}(\varphi))/(1-\alpha)
\end{align*}
where
\begin{multline*}
\psi(\varphi) :=
1/(\alpha+1) \, \sin((\alpha+1)\varphi)\cos(\alpha\omega) - \cos((\alpha+1)\varphi)\\
- 1/(\alpha-1) \, \sin((\alpha-1)\varphi)\cos(\alpha\omega) + \cos((\alpha-1)\varphi)
\end{multline*}
and \(\alpha = 856399/1572864 \approx 0.54\), \(\omega = 3\pi/2\).
To have a nonzero right-hand side we add the pressure \(p_+ := \sin(xy\pi)\), i.e. \(p := p_0 + p_+\) and \(f := \nabla(p_+)\).
We generate a pressure dominant case by using a small viscosity $\nu = 10^{-3}$.
In Figure~\ref{fig:lshape:dcontpres} and \ref{fig:lshape:contpres} the velocity error and the
novel estimator $\eta_\text{new}$ are plotted for the classical and modified version of
four different finite element methods and uniform and adaptive mesh refinement.
For this example an adaptive refinement is expected to refine the generic singularity
of the velocity in the corner $(0,0)$.

We first discuss the pressure-robust variants of the finite element methods. Looking at the left plots of Figure~\ref{fig:lshape:dcontpres} and \ref{fig:lshape:contpres} we can see that there is a major difference between adaptive and uniform mesh refinement. The adaptive algorithm results in optimal orders of the velocity error and the estimator, while uniform refinement only leads to suboptimal orders as the singularity is not resolved well enough. The only exception is the MINI finite element method which pre-asymptotically converges
  with quadratic speed. This is again thanks to the better polynomial order in the pressure ansatz space and the smooth pressure \(p^+\). Asymptotically also the MINI finite element method shows
  the suboptimal behaviour in case of uniform mesh refinement and first-order convergence in case of adaptive mesh refinement. In all cases, the new error estimator \(\mu_\text{new}\) is efficient
  and gives reasonable refinement indicators.

\begin{figure}
  \begin{center}
    \begin{tikzpicture}[scale=0.9]
      \node at (1,0) {\includegraphics[width = 0.2\textwidth, trim=100mm 20mm 100mm 20mm, clip]{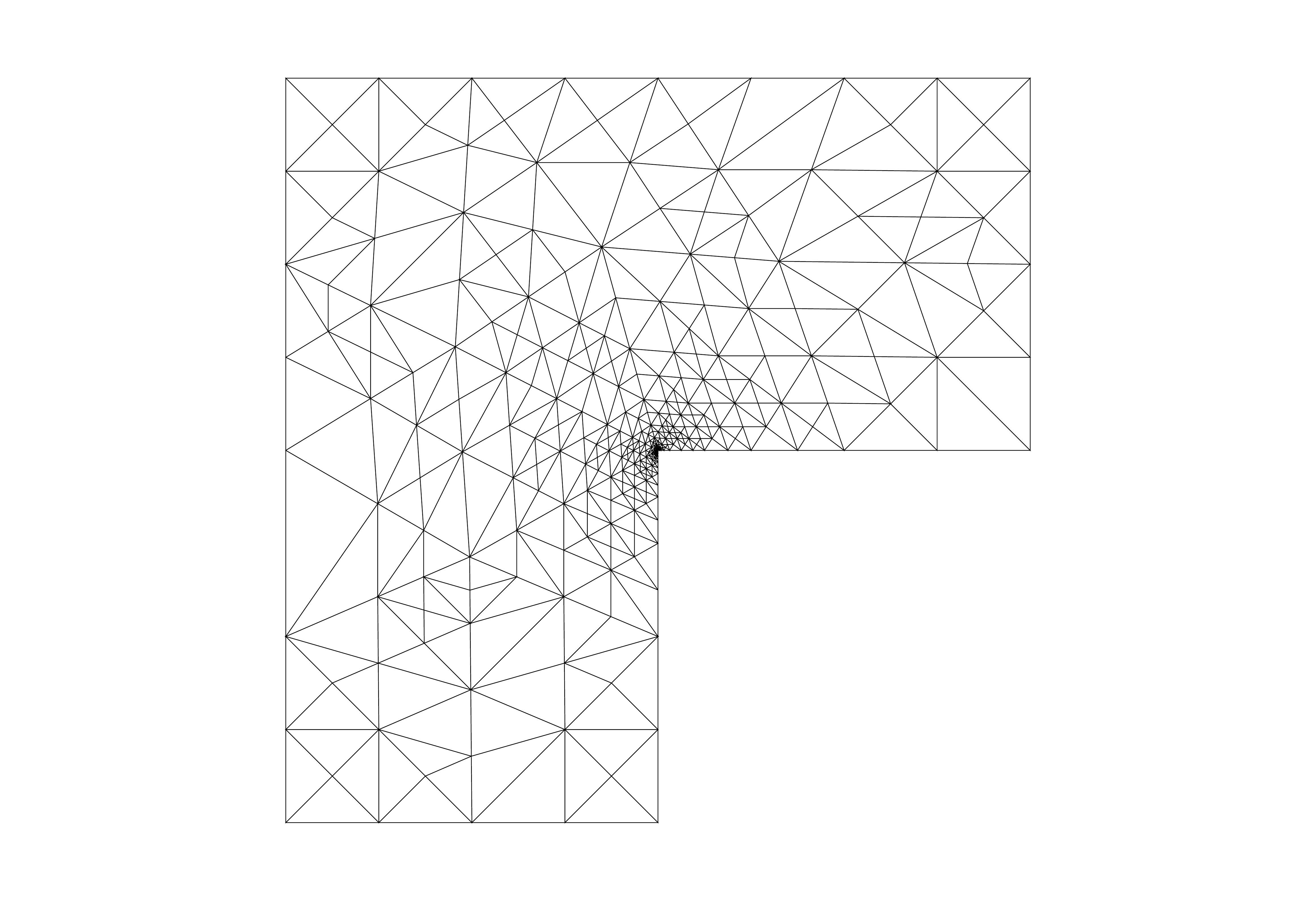}};
      \node at (5,0) {\includegraphics[width = 0.2\textwidth, trim=100mm 20mm 100mm 20mm, clip]{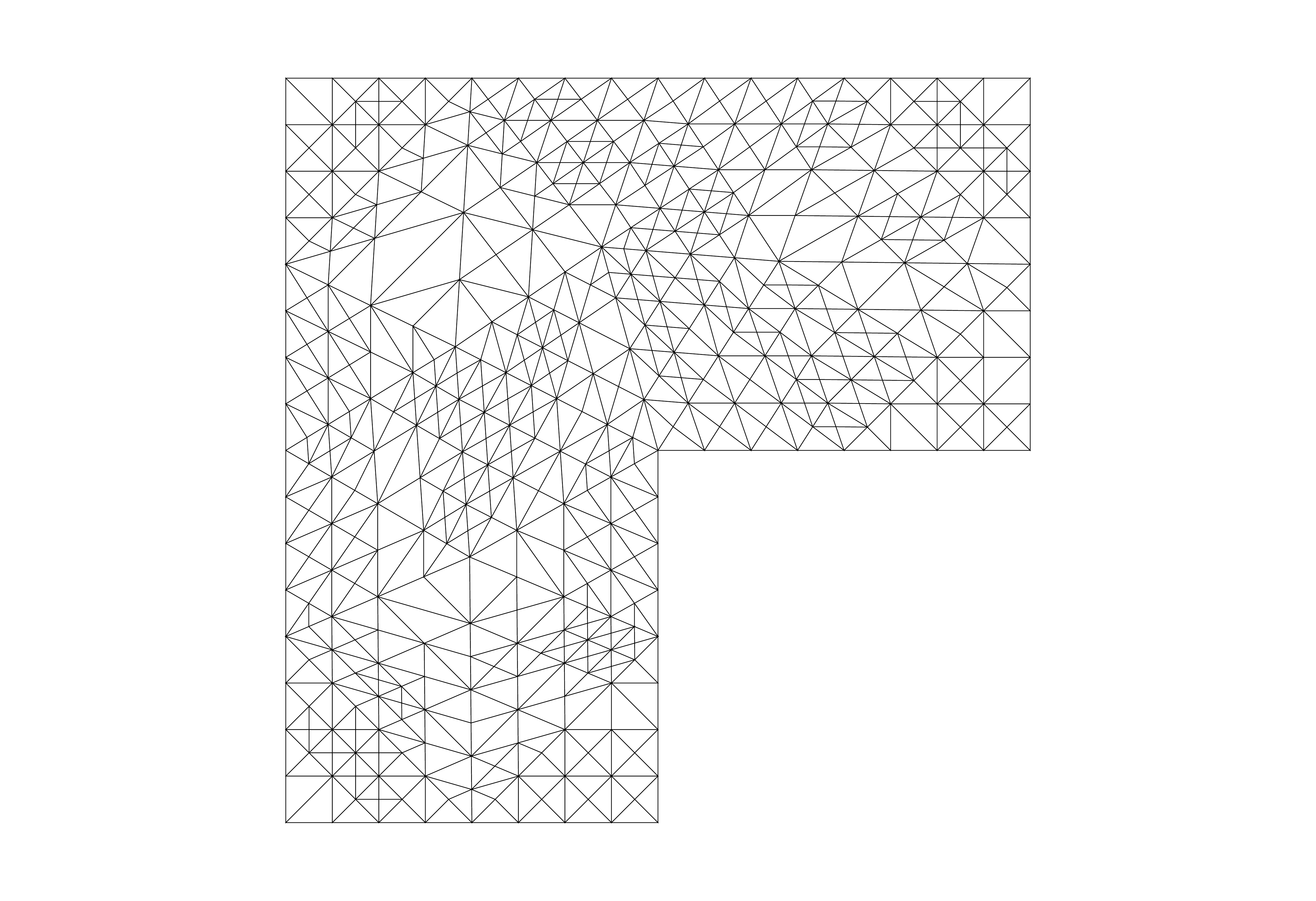}};
      \node at (9,0) {\includegraphics[width = 0.2\textwidth, trim=100mm 20mm 100mm 20mm, clip]{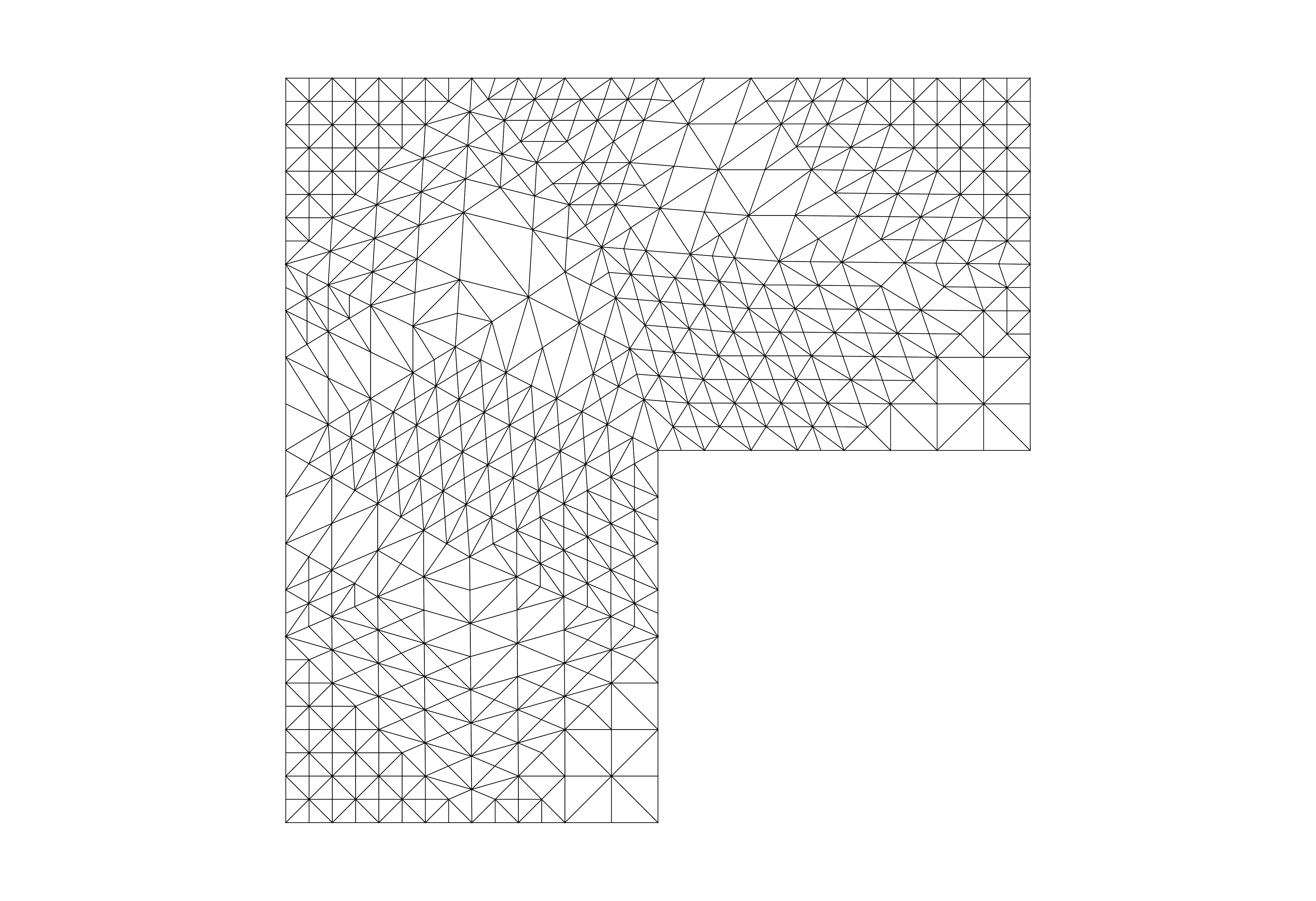}};
      \node at (13,0) {\includegraphics[width = 0.2\textwidth, trim=123mm 31mm 123mm 31mm, clip]{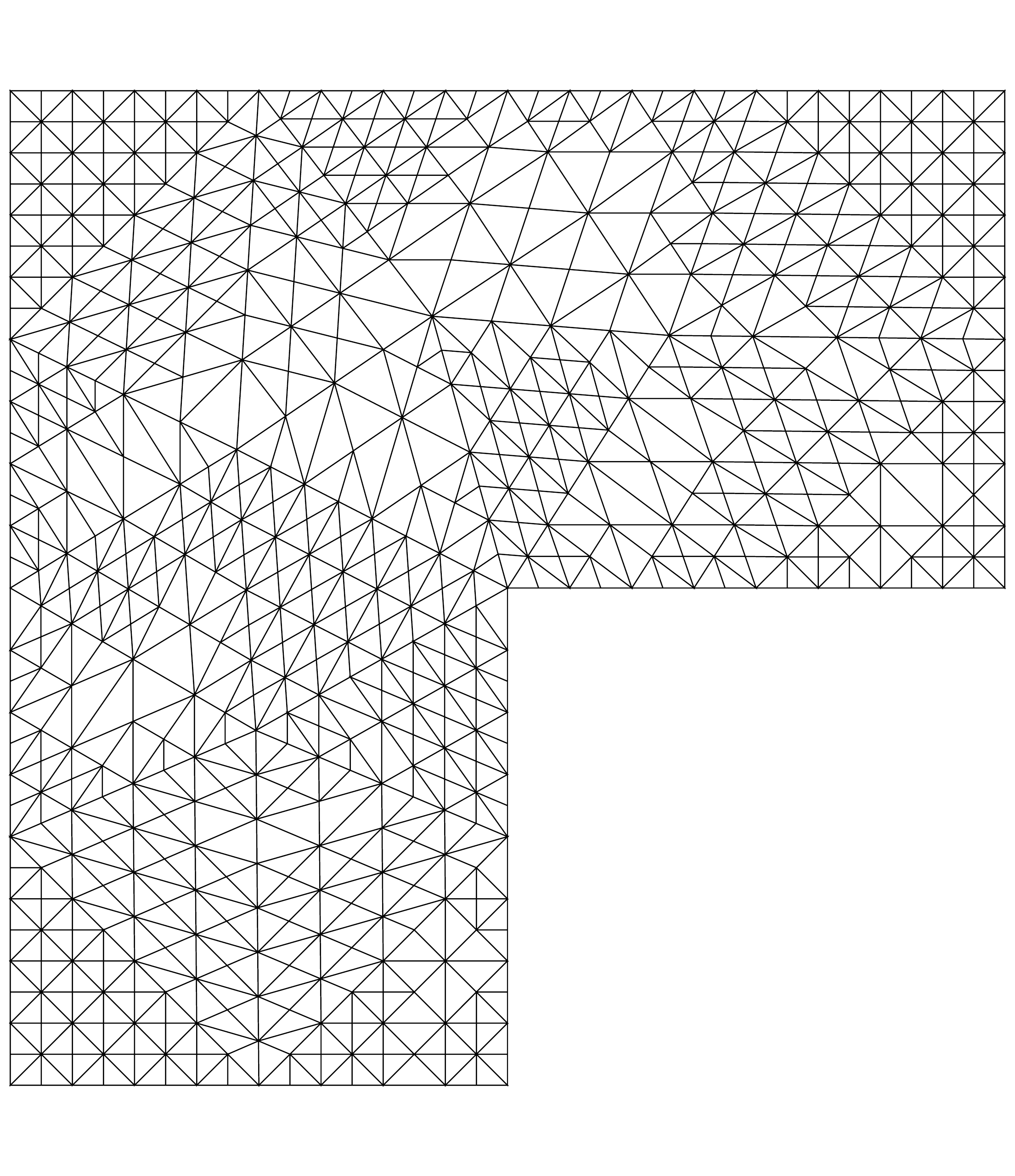}}; 
      \node at (1,-2.5) {(a)}; 
      \node at (5,-2.5) {(b)}; 
      \node at (9,-2.5) {(c)}; 
      \node at (13,-2.5) {(d)}; 
      \node[circle,draw, red, path picture ={
        \node at (path picture bounding box.center){
          \includegraphics[width = 0.55\textwidth, trim=100mm 20mm 100mm 20mm, clip]{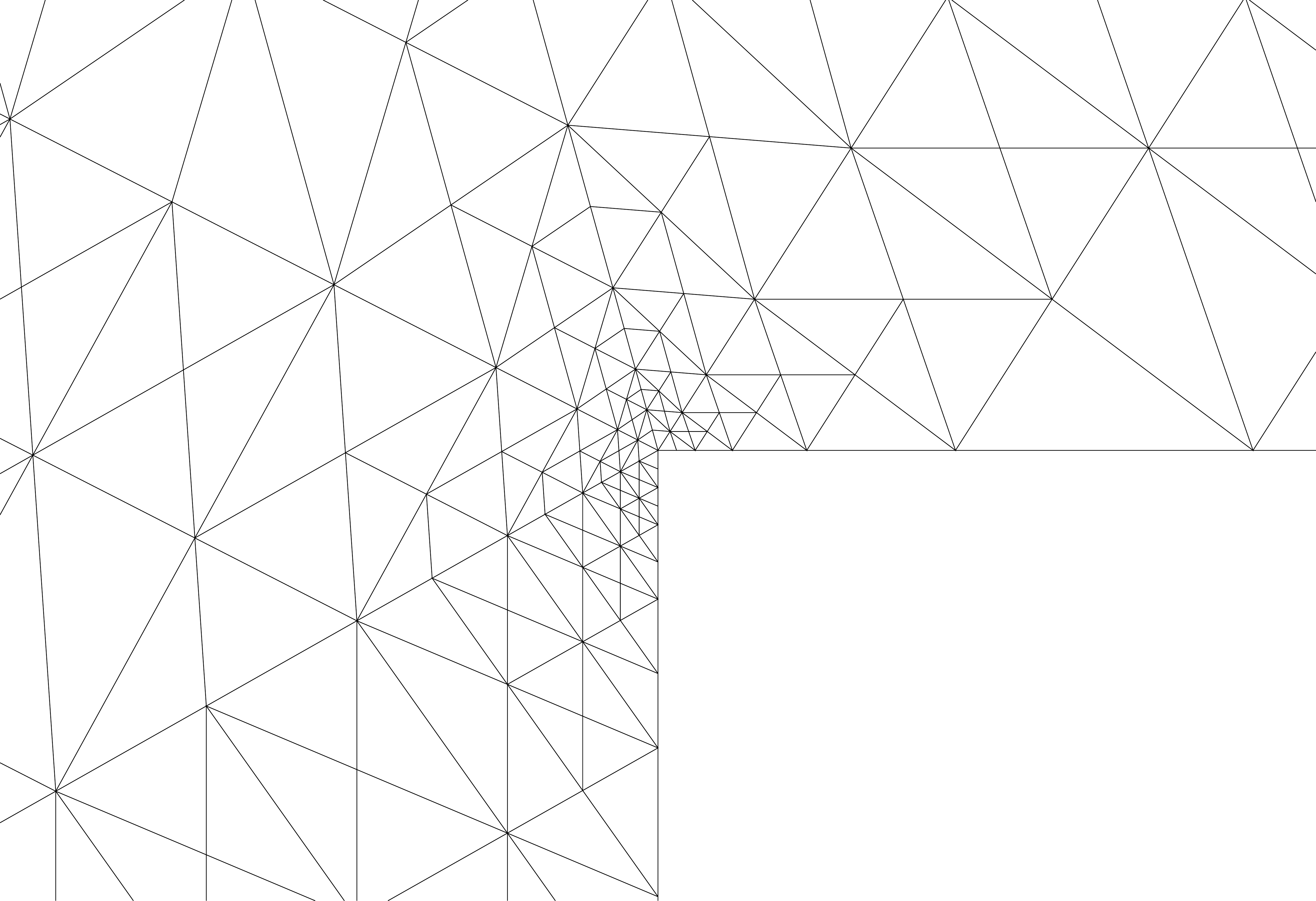}
        };
      }] at (2.0,-1) {~~~~~~~~~~~};
      \draw[red]  (1,0) circle (0.1cm) ;
    \end{tikzpicture}
  \caption{
  (a): according to \(\mu_\text{new}\) adaptively refined mesh with 4584 degrees of freedom for the pressure-robust Taylor--Hood method;
  (b): according to \(\mu_\text{new}\) adaptively refined mesh with 3971 degrees of freedom for the classical Taylor--Hood method;
  (c): according to \(\mu_\text{class}\) adaptively refined mesh with 5119 degrees of freedom for the pressure-robust Taylor--Hood method;
  (d): according to \(\mu_\text{class}\) adaptively refined mesh with 5320 degrees of freedom for the classical Taylor--Hood method
  } \label{fig:adaptref}
  \end{center}
\end{figure}
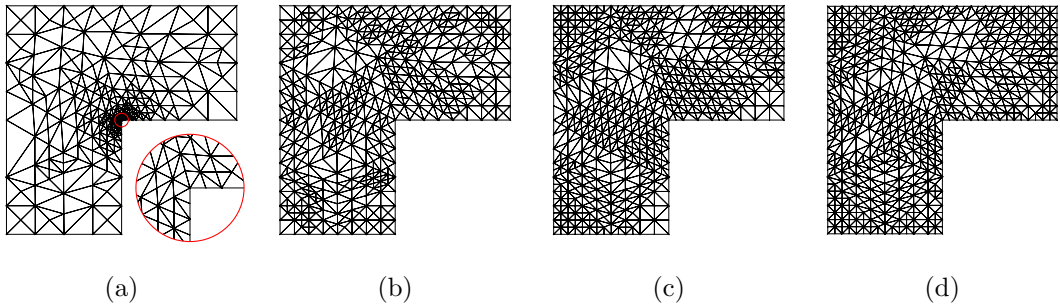

In case of the classical variants of the finite element methods, totally different observations can be made.
In the right pictures of Figure~\ref{fig:lshape:dcontpres} and \ref{fig:lshape:contpres} we first note that the error is much larger compared to the pressure-robust method.
Furthermore similar as before only adaptive mesh refinement leads to optimal orders. However, it is important to note that the gap between the velocity error of the classical method and the
velocity error of the pressure-robust method stays as large as in the beginning also under adaptive mesh refinement. A possible explanation is given by Figure~\ref{fig:adaptref}
which shows that the classical method refines the mesh almost uniformly. This is reasonable in the sense that the pressure error of the smooth
pressure \(p_+\) dominates the (real and the estimated) discretisation error in the beginning.
The pressure-robust method on the other hand is not polluted by this influence and can concentrate immediately on the corner singularity. However, it is important that the
error estimator is also pressure-robust. If the refinement indicators are taken from \(\mu_\text{class}\), the corner singularity remains unrefined until the dominance of the pressure
error in the error bound is removed.
Hence, the main conclusion is that only a pressure-robust finite element method with a pressure-robust error estimator leads to optimal
meshes with the smallest velocity error.

\bibliographystyle{plain}
\bibliography{lit}

\end{document}